\documentclass{siamart190516}
\usepackage{amsmath,amssymb,mathdots,algorithm,algorithmic,float,graphicx}
\usepackage{caption}
\usepackage{subcaption}
\usepackage[utf8]{inputenc}
\newcommand{\mat}[1]{\boldsymbol{#1}}
\newcommand{\matc}[1]{\boldsymbol{\mathcal{#1}}}
\newcommand{\mathat}[1]{\boldsymbol{    \widehat{#1}}}
\newcommand{\mattilde}[1]{\widetilde{\boldsymbol{#1}}}
\newcommand{\mc}[1]{\mathcal{#1}}
\newcommand{\mb}[1]{\mathbb{#1}}
\newcommand{\bmat}[1]{\begin{bmatrix} #1
\end{bmatrix}}
\newcommand{\B}[1]{\boldsymbol{#1}}
\newcommand{\range}{\mathcal{R}\,}

\usepackage{verbatim}
\newtheorem{assumption}{Assumption}

\newcommand{\Ah}{\B{\widehat{A}}}
\newcommand{\Uh}{\B{\widehat{U}}}
\newcommand{\Vh}{\B{\widehat{V}}}
\newcommand{\Sh}{\B{\widehat{\Sigma}}}
\newcommand{\Hh}{\B{\widehat{\mc{H}}}}

\newcommand{\sv}{\mathsf{sv}\,}
\newcommand{\proj}[1]{\B{\mc{P}}_{#1}}
\newcommand{\rank}{\mathsf{rank}\,}

\usepackage{ulem,color}

\title{Efficient Algorithms for Eigensystem Realization using Randomized SVD\thanks{Submitted to the editors on TBD.\funding{This work was funded, in part, through the National Science Foundation: R.M. was supported by DMS-1745654, A.K.S. was supported by DMS-1821149, J.K and A.C. were  supported by ECS-1711004. }}}
\author{Rachel Minster\thanks{Department of Mathematics, North Carolina State University (\email{rlminste@ncsu.edu}, \email{asaibab@ncsu.edu})} \and  Arvind K. Saibaba\footnotemark[2] \and Jishnudeep Kar\thanks{Department of Electrical and Computer Engineering,   North Carolina State University (\email{jkar@ncsu.edu}, \email{achakra2@ncsu.edu})} \and Aranya Chakrabortty\footnotemark[4]}
\begin{document}
\maketitle
\begin{abstract} Eigensystem Realization Algorithm (ERA) is a data-driven approach for subspace system identification and is widely used in many areas of engineering. However, the computational cost of the ERA is dominated by a step that involves the singular value decomposition (SVD) of a large, dense matrix with block Hankel structure. This paper develops computationally efficient algorithms for reducing the computational cost of the SVD step by using randomized subspace iteration and exploiting the block Hankel structure of the matrix. We provide a detailed analysis of the error in the identified system matrices and the computational cost of the proposed algorithms. We demonstrate the accuracy and computational benefits of our algorithms on two test problems: the first involves a partial differential equation that models the cooling of steel rails, and the second is an application from power systems engineering.  
\end{abstract}


\section{Introduction}
{Linear Time Invariant (LTI) systems form one of the most important classes of dynamic systems existing in the physical world. LTI systems closely approximate the linear behavior of any nonlinear physical process around a desired operating point, and provide insights on how the system will behave in response to any small-signal change in its equilibrium state. The state-space representation of a LTI discrete-time system can be written as 
\begin{equation}\label{eqn:dynamics}
\begin{aligned}
\mat{x}_{k+1} &=\mat{Ax}_k+\mat{Bu}_k, \qquad \mat{x}_0 = \mat{x}(0) \\
\mat{y}_k &= \mat{Cx}_k+\mat{Du}_k,
\end{aligned}
\end{equation}
where $n$ is the order of the system, $k=0,\,1,\,2,\dots$ is the sampling index, $\mat{x}_k \in \mb{R}^{n}$ is the state vector, $\mat{u}_k \in \mb{R}^{m}$ is the input vector, and $\mat{y}_k \in \mb{R}^{\ell}$ is the output vector at sampling instant $k$, with $m$ and $\ell$ being the number of inputs and outputs, respectively. The matrix $\mat{A} \in \mb{R}^{n \times n}$ is referred to as the state matrix, $\mat{B} \in \mb{R}^{n \times m}$ is the input matrix, $\mat{C} \in \mb{R}^{\ell \times n}$ is the output matrix, and $\mat{D} \in \mb{R}^{\ell \times m}$ is the input-output feedthrough matrix. In the majority of practical applications, however, these four matrices may not be exactly known to the system modeler because of different kinds of model uncertainties and operational uncertainties. In that case, one must estimate these matrices from sampled measurements of the input sequence $\{\mat{u}_k\}$ and the output sequence $\{\mat{y}_k\}$. This process is referred to as system identification. Since the state may be represented in different coordinate frames without changing input-output characteristics, the more formal definition of system identification is given as the process of estimating $(\mat{A},\mat{B},\mat{C},\mat{D})$ up to a similarity transformation $(\mat{TAT}^{-1},\mat{TB},\mat{CT}^{-1},\mat{D})$, where $\mat{T}\in \mb{R}^{n\times n}$ is an invertible matrix. As mentioned, such a transformation does not change the input-output behavior or the transfer function of the system. We assume that the system to be identified is observable and reachable \cite{ljung}.  There exists a vast literature on system identification with a variety of numerical algorithms, developed for various applications such as electric power systems, process control, mechanical systems, aerospace applications, civil and architectural applications, and chemical and biological processes to name a few. For a survey of these methods, please see ~\cite{antoulas2005approximation,ljung,van2012subspace,verhaegen2007filtering}}. In this paper, we focus on identifying the system matrices when the inputs $\{\mat{u}_k\}$ are in the form of impulse functions.

One of the most common identification methods used for identifying LTI models in practice is known as the Eigensystem Realization Algorithm (ERA), initially proposed in \cite{kung1978new}, but can be found in many other sources such as \cite{verhaegen2007filtering}. ERA is a subset of a broader class of identification methods called subspace system identification, and is a purely data-driven approach consisting of two main steps. In the first step, this algorithm computes the singular value decomposition (SVD) of a block Hankel matrix, constructed using the impulse response data. In the second step, the truncated SVD is used to solve a least-squares optimization problem to identify the system matrices. The first step is computationally expensive since it has cubic complexity in the dimensions of the matrix. To reduce the computational cost, the tangential interpolation-based ERA (TERA)~\cite{kramer2016tangential} was proposed. This method reduces the number of inputs and outputs, thus reducing the dimensions of the matrix whose SVD needs to be computed. Another algorithm, known as CUR-ERA~\cite{kramer2018system}, uses the CUR decomposition to efficiently compute the low-rank decomposition of the block Hankel matrix. While these algorithms have successfully lowered the costs, important challenges remain.

\paragraph{Contributions and Content} We propose new randomized algorithms for tackling the computational costs of ERA. The first algorithm (Section~\ref{sec:randera}) accelerates the standard randomized SVD by exploiting the block Hankel structure to efficiently perform matrix-vector products. The second algorithm (Section~\ref{ssec:randtera}) employs similar ideas as the first algorithm in combination with the TERA. The resulting algorithms are efficient both in terms of storage costs and computational costs. In Section~\ref{sec:err}, we derive new error bounds that provide insight into the accuracy and stability of the system matrices identified using the approximation algorithms. The error analysis is not tied to any particular algorithm and in this sense is fairly general. Finally, we demonstrate the benefits of the proposed algorithms on two different applications: first, a heat transfer problem for cooling of steel rails, and second, a dynamic model of an electric power system with multiple generators and loads.

\paragraph{Comparison to related work} While TERA works with smaller matrices, it still has cubic complexity in the number of time measurements. This is computationally demanding when the dynamics of the system are slow, and many measurements in time are needed to fully resolve the dynamics. The CUR-ERA algorithm has similar storage and computational complexity as the algorithms we propose but numerical experiments (see Section~\ref{sec:num}) suggest that our algorithm is more accurate.  Furthermore, the randomized algorithms developed here can exploit parallelism in multiple ways; the matrix-vector products can be parallelized across multiple random vectors as well as over the input/output pairs. This makes it attractive for implementations on high performance computing platforms. If the target rank for the low-rank decomposition is not known in advance, one can use randomized range finding algorithms~\cite{martinsson2016randomized,yu2018efficient} to estimate the target rank.  Two other features make our contributions attractive: first, the block Hankel structure that we exploit here can be adapted to any other matrix-free low-rank approximation algorithm, for example, based on the Lanczos bidiagonalization~\cite{simon2000low}, and second, the error analysis developed here is informative for any approximation algorithm. There are other randomized algorithms for system identification~\cite{wu2016fast,yu2016computationally} but they tackle slightly different settings.

\section{Background}
In this section, we first review the Eigensystem Realization Algorithm (Section~\ref{sec:era}), the associated computational costs, and motivate the need for efficient algorithms.  We also review the key ingredients needed to construct our algorithms: algorithms for storing and computing with Hankel matrices (Section~\ref{ssec:hankel}), and randomized SVD (Section~\ref{ssec:randsvd}) for computing the low-rank factorizations.
\subsection{Eigensystem Realization Algorithm}\label{sec:era}
ERA was first proposed in \cite{kung1978new}, but we follow the formulation in \cite{kramer2016tangential}. Other references for the ERA include \cite{verhaegen2007filtering,van2012subspace}. Now, assuming the initial condition $\mat{x}_0 = \mat{0}\in \mb{R}^n$, and given a sequence of inputs $\mat{u}_i$ for $i= 1,\dots$, the outputs observed are 
\begin{equation}\label{eqn:obs}\mat{y}_k = \sum_{j=0}^{k-1} \mat{h}_{k-j-1}\mat{u}_{j} \quad k = 0,1,\dots,\end{equation} where the matrices $\{\mat{h}_k\}$ are called the \textit{Markov parameters} and are given by
\begin{equation}\label{eqn:markov}
    \mat{h}_k = \left\{ \begin{array}{ll}
       \mat{D} &  k = 0 \\
    \mat{CA}^{k-1}\mat{B} & k = 1,\dots,2s-1. 
    \end{array}\right.
\end{equation}

 The system is excited $m$ times using impulse input excitations. That is, we take 
\[\mat{u}_k^{(j)} = \left\{ \begin{array}{ll}\mat{e}_j & k =0 \\ \mat{0} & k > 0\end{array} \right. \qquad j =1,\dots,m. \] 
Here $\mat{e}_j$ is the $j$th column of the $m\times m$ identity matrix. The outputs from each impulse excitation can be used to construct the Markov parameters $\{ \mat{h}_k \}$. The ERA uses the Markov parameters to recover the system matrices. If data using impulse inputs is not available, then one can estimate the Markov parameters from the general input data~\cite{juang1993identification}.

These Markov parameters are first arranged to form the block Hankel matrix $\matc{H}_{s} \in \mb{R}^{(s\ell) \times (sm)}$ defined as 
\begin{equation}\label{eqn:H}
    \begin{aligned}
\boldsymbol{\mc{H}}_{s} &= \bmat{\mat{h}_1 & \mat{h}_2 & \dots & \mat{h}_s\\ \mat{h}_2 & \mat{h}_3 & \dots & \mat{h}_{s+1} \\ \vdots & \vdots & \iddots & \vdots \\ \mat{h}_s & \mat{h}_{s+1} & \dots & \mat{h}_{2s-1}} \\[8pt]
&= \bmat{\mat{CB} & \mat{CAB} & \dots & \mat{CA}^{s-1}\mat{B}  \\ \mat{CAB} & \mat{CA}^2\mat{B} & \dots & \mat{CA}^s\mat{B} \\ \vdots & \vdots & \iddots & \vdots  \\ \mat{CA}^{s-1}\mat{B} & \mat{CA}^s\mat{B} & \dots & \mat{CA}^{2s-2}\mat{B} }. 
\end{aligned}
\end{equation}
Here, $s$ is a chosen parameter that determines the number of block rows and columns of $\matc{H}_{s}$. It is known that this matrix can be factorized into the observability matrix $\matc{O}_s$ and the controllability matrix $\matc{C}_s$, as $\matc{H}_{s} = \matc{O}_s\matc{C}_s$, where 
\[ \matc{O}_s = \bmat{ \mat{C} \\ \mat{CA} \\ \vdots \\ \mat{CA}^{s-1}}, \qquad \matc{C}_s =\bmat{ \mat{B} & \mat{AB} &  \dots & \mat{A}^{s-1}\mat{B}}.\]
If the system is assumed to be reachable and observable, then $\range(\matc{H}_{s}) = \range(\matc{O}_s)$, where $\range(\cdot)$ denotes the range (or column space) of a matrix. We can obtain a basis for $\matc{O}_s$ using  the reduced SVD of $\matc{H}_{s} = \mat{U}_n\mat{\Sigma}_n\mat{V}_n^\top$. Then, partition the left singular vectors as 
\[ \mat{U}_n = \bmat{\mat\Upsilon_f \\ \mat{*}} = \bmat{\mat{*} \\ \mat\Upsilon_l},\]
where $\mat\Upsilon_f ,\mat\Upsilon_l\in \mb{R}^{(s-1)\ell\times n}$. We can obtain $\mat{A}$ using the formula $\mat{A} = \mat\Upsilon_f^\dagger \mat\Upsilon_l$, where $^\dagger$ denotes the Moore-Penrose pseudoinverse. Furthermore, we can obtain the output matrix $\mat{C}$ and the input matrix $\mat{B}$ using the formulas
\begin{equation}\label{eqn:BC} \mat{C} = \bmat{\mat{I}_{\ell} & \mat{0}} \mat{\Upsilon}_f, \qquad  \mat{B} = \mat{\Sigma}_n\mat{V}_n^\top\bmat{\mat{I}_{m} \\ \mat{0}}. \end{equation}
The matrix $\mat{D}$ can be identified as the Markov matrix $\mat{h}_0$ and, therefore, we will not discuss its estimation in future sections.

\paragraph{Reduced order model}In some applications, the goal is not only to identify the system but also obtain a reduced order model of the system. That is, we seek the reduced system matrices $\mat{A}_r \in \mb{R}^{r\times r}$, $\mat{B}_r \in \mb{R}^{r \times m}$, $\mat{C}_r\in \mb{R}^{\ell\times r}$ and $\mat{D}_r \in \mb{R}^{\ell \times m}$ which approximate the dynamics of the original system~\eqref{eqn:dynamics}.
To accomplish this, as before, we first compute a rank $r$ approximation to the matrix $\matc{H}_s$ as 
\[ \matc{H}_s \approx \mat{U}_r\mat{\Sigma}_r\mat{V}_r^\top,\]
where $r \leq n$. We partition the left singular vectors $\mat{U}_r$ as 
\[ \mat{U}_r = \bmat{\mat\Upsilon_f^{(r)} \\ \mat{*}} = \bmat{\mat{*} \\ \mat\Upsilon_l^{(r)}},\]
such that $\mat\Upsilon_f^{(r)} \in \mb{R}^{(s-1)\ell\times r}$ and $\mat\Upsilon_l^{(r)} \in \mb{R}^{(s-1)\ell \times r}$. Then, we compute the reduced order model $\mat{A}_r = \left[\mat\Upsilon_f^{(r)}\right]^\dagger \mat\Upsilon_l^{(r)}$ such it that minimizes the least squares problem 
\[ \min_{\mathat{A} \in \mb{R}^{r\times r}} \|\mat\Upsilon_f^{(r)}\mathat{A} - \mat\Upsilon_l^{(r)} \|. \] The reduced order output matrix $\mat{C}_r$ and the input matrix $\mat{B}_r$ are computed using the formulas
\begin{equation}\label{eqn:BCr} \mat{C}_r = \bmat{\mat{I}_{\ell} & \mat{0}} \mat{U}_r, \qquad  \mat{B}_r = \mat{\Sigma}_r\mat{V}_r^\top\bmat{\mat{I}_{m} \\ \mat{0}}. \end{equation}

Other papers, such as \cite{kramer2018system}, use a slightly different representation for the system matrix $\mat{A}_r$ than the one used in~\cite{kung1978new}. In this alternate representation, $\mat{A}_r'$ is obtained using $\mat{A}_r' = \mat{\Sigma}_r^{-1/2}\mat{\Upsilon}_f^\dagger \mat{\Upsilon}_l \mat{\Sigma}_r^{1/2}$. Then, $\mat{B}_r'$ and $\mat{C}_r'$ are obtained using $$\mat{C}'_r= \bmat{\mat{I}_{\ell} & \mat{0}} \mat{\Upsilon}_f \mat{\Sigma}_r^{1/2}, \quad \mat{B}_r' = \mat{\Sigma}_r^{1/2} \mat{V}_r^\top \bmat{\mat{I}_m \\ \mat{0}}.$$ 

Note that the two formulations are equivalent up to a similarity transformation, and this results in the same Markov parameters and input-output behavior of the system. If $r = n$, then the ERA determines the system matrices $(\B{A},\B{B},\B{C},\B{D})$ and there is no model reduction step. On the other hand, if the target rank $r < n$, then the reduced order model is guaranteed stability under the conditions of~\cite[Theorem 3]{kramer2018system}.

\paragraph{Computational Cost} We briefly review the range of possible parameters. \begin{enumerate}
    \item State space: In the power system applications, the dimension of the state space $n$ is around $10^2$, whereas in applications with partial differential equations this dimension can be large, i.e., $10^5-10^7$ (see for example,~\cite{ma2011reduced}).
    \item Input dimension: The number of inputs $m$ is in the range $1-10^2$. 
    \item Output dimension: The number of outputs $\ell$ is also in the range $1-10^{2}$. In certain applications, the number of outputs is also the dimension of the state space. 
    \item Sample size: When the dynamics of the system are slow, the number of samples $2s-1$ can be large, i.e. $10^2-10^5$.
\end{enumerate}
The dominant cost of ERA is the cost of storing and factorizing the matrix $\matc{H}_s$, which is of size $s\ell \times sm$. The cost of storing the matrix $\matc{H}_s$ is $\mc{O}(s^2\ell m)$, whereas the cost of computing the SVD is $\mc{O}(s^3\ell m \min\{\ell,m\})$. In~\cite{kramer2018system,kramer2016tangential}, the authors discuss applications in which the size of $\matc{H}_s$ is $O(10^4)$. However, for the higher end of the parameters listed above, the matrix $\matc{H}_s$ is of size ${O}(10^7)$. Forming and factorizing this matrix is completely infeasible. The algorithms we propose are both efficient in storage and computational cost and make ERA applicable to larger problem sizes. 

\subsection{Hankel matrices}\label{ssec:hankel}
Structured matrices such as circulant and Hankel matrices are computationally efficient to work with since matrix-vector products (matvecs) can be accelerated using the Fast Fourier Transform (FFT). We first review circulant matrices. Circulant matrices are completely determined by their first column $$ \mat{x} = \bmat{ x_1 & \dots & x_N}^\top,$$ and are diagonalized by the Fourier matrix. Let the circulant matrix $\mat{X}$ be defined as
\[\mat{X} = \bmat{x_1 & x_N & \dots & x_2& \\ x_2 & x_1 & \dots & x_3 \\ \vdots & \vdots & \ddots & \vdots \\ x_N & x_{N-1} & \dots & x_1 },\]
and let $\mat{F}_N$ be the $N\times N$ Fourier matrix with entries $(\mat{F}_N)_{jk} = e^{2\pi ijk/n}$. Then the eigenvalue decomposition of $\mat{X}$ is $\mat{X} = \mat{F}_N^*\text{diag}(\mathbf{F}_N\mat{x}) \mathbf{F}_N$, where $\mat{x} = \mat{X}(:,1)$ is the first column of $\mat{X}$. This means that the circulant matrix has eigenvalues $\mathbf{F}_N\mat{x}$. This result implies that matvecs $\mat{Xv}$ can be performed efficiently using FFTs as $\mat{y} = \text{IFFT}( \text{FFT}(\mat{v}) \odot \text{FFT}(\mat{x}))$, where $\odot$ is the elementwise product, FFT$(\cdot)$ and IFFT$(\cdot)$ denote the fast Fourier and inverse fast Fourier transforms respectively. The computational cost involves 2 FFTs and one inverse FFT, and can be implemented efficiently in $\mc{O}(N\log_2 N)$ floating point operations (flops)

Hankel matrices have constant entries along every anti-diagonal; that is, given the parameters $h_{1},h_2,\dots, h_{2s-1},$
the Hankel matrix is 
\[ \mat{H}_s = \bmat{h_1 & h_2 & \dots & h_s \\ h_2 & h_3 & \dots & h_{s+1} \\ \vdots & \vdots & \iddots & \vdots \\ h_s & h_{s+1} & \dots & h_{2s-1} }.\]
This corresponds to the single input/single output (SISO) case. Note that permuting a Hankel matrix with the reverse identity permutation matrix $\mat{J}_s$ results in a Toeplitz matrix (constant entries along every diagonal). That is, 
\[ \mat{J}_s = \bmat{ 0 & 0 & \dots & 0 &  1 \\0 & 0 & \dots & 1 & 0 \\\vdots & \vdots & \iddots & \vdots & \vdots \\ 0 &  1 & \dots & 0 & 0 \\  1 & 0 & \dots & 0 & 0  }, \qquad \mat{H}_s\mat{J}_s = \bmat{h_s & h_{s-1} & \dots & h_2 &  h_1 \\h_{s+1}  & h_s  & \dots & h_3 & h_2 \\\vdots & \vdots & \ddots & \vdots & \vdots \\ h_{2s-2} &  \vdots & \dots & h_s & h_{s-1}  \\  h_{2s-1} & h_{2s-2} & \dots & h_{s+1} & h_s }.\]
We can use this to our advantage while computing matrix-vector products with Hankel matrices. To compute the matrix-vector product $\mat{y} = \mat{H}_s\mat{v}$, we first write $\mat{y} = (\mat{H}_s\mat{J}_s)(\mat{J}_s\mat{v})$ so that 
\[ \mat{X}_{2s}\bmat{\mat{J}_s\mat{v} \\ \mat{0}_s} = \bmat{\mat{H}_s\mat{J}_s & \mat{B}_s \\ \mat{B}_s & \mat{H}_s\mat{J}_s} \bmat{\mat{J}_s\mat{v} \\ \mat{0}_s} = \bmat{\mat{y} \\ \mat{*}}, \]
where $\mat{*}$ denotes the part of a computation which we can ignore and 
\[ \mat{B}_s = \bmat{ 0 & h_{2s-1} & \dots  & h_{s+2} & h_{s+1} \\ h_1 & 0 & h_{2s-1} & \dots &  h_{s+2} \\ \vdots & \vdots & \ddots & \vdots & \vdots \\ h_{s-2} &  \vdots & \dots & 0 & h_{2s-1}  \\  h_{s-1} & h_{s-2} & \dots & h_{1} & 0 } .\]

This means that we can also efficiently compute matvecs with Hankel matrices by embedding it within a $2s \times 2s$ circulant matrix $\B{X}_{2s}$ defined by the vector 
\[ \mat{x}_{2s} = \bmat{ h_s & h_{s+1} & \dots & h_{2s-1} & 0 &  h_{1} & h_{2} & \dots & h_{s-1}}^\top.\]
Thus, only the first row and the last column need to be stored to compute matrix-vector products with the Hankel matrix $\mat{H}_s$. A summary of the algorithm to compute matvecs with $\mat{H}_s$ is given in Algorithm~\ref{alg:hankelmult2}.
\begin{algorithm}[!ht]
\begin{algorithmic}[1]
\REQUIRE last column  $\mat{h}_c \in \mb{R}^s$,  first row $\mat{h}_r \in \mb{R}^s$ of a Hankel matrix $\mat{H}_s$, vector $\mat{v} \in \mb{R}^s$
\ENSURE matvec $\mat{y} = \mat{H}_s\mat{v} \in \mb{R}^s$
\STATE Form circulant vector $\mat{x} = \bmat{\mat{h}_c^\top &0 & \mat{h}_r(1:\text{end}-1)}^\top$  
\STATE Pad the vector $\mat{v}$ to get $\mat{\hat{v}} = \bmat{\mat{v} & \mat{0}_{s}}^\top$
\STATE Take $\mat{z} = \text{IFFT}(\text{FFT}(\mat{x}) \odot \text{FFT}(\mat{\hat{v}}))$. 
\STATE Extract $\mat{y} = \mat{z}(1:s)$
\end{algorithmic}
\caption{$\mat{y} = $ Hankel-matvec$(\mat{h}_c,\mat{h}_r,\mat{v})$}
\label{alg:hankelmult2}
\end{algorithm}

\subsection{Randomized SVD}\label{ssec:randsvd}
We can efficiently compute a rank $r$ approximation of an $M \times N$ matrix $\mat{X}$ using a randomized version of the SVD~\cite{halko2011finding} (henceforth called RandSVD). The idea is to find a matrix $\mat{Q}$ that approximates the range of $\mat{X}$.  This is done by first drawing a Gaussian random matrix $\mat{\Omega} \in \mb{R}^{N \times (r+\rho)}$, where $r$ is the desired target rank, and $\rho \geq 0$ is an oversampling parameter (typically $\rho \leq 20$).  Then, the matrix $\mat{Y} = \mat{X\Omega}$ consists of random linear combinations of the columns of $\mat{X}$.  This means that we can get a matrix $\mat{Q}$ such that $\range(\mat{X}) \approx \range(\mat{Q})$ by taking a thin QR factorization $\mat{Y} = \mat{QR}$.  If $\mat{X}$ has singular values that decay rapidly, or $\rank(\mat{X})$ is exactly $r$, then $\range(\mat{Q})$ is a good approximation for $\range(\mat{X})$.  We can then approximate $\mat{X}$ by the low-rank representation $\mat{X} \approx \mat{Q}\mat{Q}^\top \mat{X}$; this can then be converted into the appropriate SVD format. The procedure is summarized in Algorithm~\ref{alg:randsvd}.

\begin{algorithm}[H]
\begin{algorithmic}[1]
\REQUIRE matrix $\mat{X} \in \mb{R}^{M \times N}$ with target rank $r$, oversampling parameter $\rho$ such that $r+\rho \leq \min\{M,N\}$, and number of subspace iterations $q\geq0$ \\
\ENSURE $\mathat{U} \in \mb{R}^{M \times r}$, $\mathat{\Sigma} \in \mb{R}^{r \times r}$, and $\mathat{V} \in \mb{R}^{N \times r}$ 
\STATE Draw Gaussian matrix $\mat{\Omega} \in \mb{R}^{N \times (r+\rho)}$ 
\STATE Multiply $\mat{Y} = (\mat{XX}^\top)^q\mat{X}\mat{\Omega}$
\STATE Thin QR factorization $\mat{Y} = \mat{Q}\mat{R}$
\STATE Form $\mat{B} = \mat{Q}^\top \mat{X}$
\STATE Calculate SVD $\mat{B} = \mat{U}_{\mat{B}} \mat{\Sigma} \mat{V}^\top$
\STATE Set $\mathat{U} = \mat{Q}\mat{U}_{\mat{B}}(:,1:r)$, $\mathat{\Sigma} = \mat{\Sigma}(1:r,1:r)$, and $\mathat{V} = \mat{V}(:,1:r)$. 

\end{algorithmic}
\caption{$[\mathat{U},\mathat{\Sigma},\mathat{V}]$=RandSVD$(\mat{X},r,\rho)$}
\label{alg:randsvd}
\end{algorithm}

RandSVD is computationally beneficial compared to the full SVD. In this paper, we use a variation of the RandSVD that uses $q\geq 0$ steps of the subspace iteration~\cite[Algorithm 4.4]{halko2011finding}. Note that for numerical stability,  we perform orthogonalization during and in between the subspace iterations. Assuming $M\geq N$, the cost of the full SVD is $\mc{O}(MN^2)$. On the other hand, if the cost of a matrix-vector product with $\mat{X}$ (or its transpose) is $T_{\mat{X}}$, then the cost of the randomized SVD can be expressed as 
\begin{equation}\label{eqn:randcost}
\text{Cost} = (2q+1)(r+\rho)T_{\mat{X}} + \mc{O}(r^2(M+N)).\end{equation}
We will use RandSVD in different ways in the system identification algorithms that we derive. We have found RandSVD with $q=0-2$ subspace iterations to be computationally efficient and sufficiently accurate for our purpose; however, there are several new randomized algorithms developed that have been reviewed in the recent paper~\cite{martinsson2020randomized}.

\section{Randomized algorithms for Eigensystem Realization}
In this section, we derive two randomized algorithms for efficient computation of the system matrices. The first algorithm accelerates the standard randomized SVD using the block Hankel structure of $\matc{H}_s$ (Section~\ref{sec:randera}); the second algorithm is a randomized variant of the Tangential Interpolation-based ERA (TERA) and is applicable when the number of inputs and outputs are large. 
\subsection{Randomized Eigensystem Realization Algorithm}\label{sec:randera}
Our first approach accelerates the computation of the system matrices by combining two ingredients: first, we replace a reduced SVD of $\matc{H}_{s}$ by a RandSVD to obtain an approximate basis for $\matc{O}_s$; second, we additionally exploit the block Hankel structure of $\matc{H}_{s}$ to accelerate the matvecs involving $\matc{H}_{s}$ and $\matc{H}_{s}^\top$ in the RandSVD. As we will show, each of these steps decreases the computational complexity yielding an efficient algorithm overall. 

\subsubsection{Block Hankel Matrices} \label{sssec:blockHankel}
We first explain how we exploit the block Hankel structure of the matrix $\matc{H}_{s}$ defined in~\eqref{eqn:H}. The multiplication process for Hankel matrices can be extended to block Hankel matrices. Suppose we have to compute $\mat{y} = \matc{H}_s\mat{x}$ for a given vector $\mat{x}$. Let us define the index sets 
\[\begin{aligned}
\mc{I}_i = &\>\{i,i+\ell, \dots, i+ (s-1)\ell\} \qquad & i = 1,\dots,\ell\\
\mc{J}_j = & \>\{j,j+m, \dots, j+ (s-1)m\} \qquad & j = 1,\dots,m.
\end{aligned}
\]
If we denote $\matc{H}_s(\mc{I}_i,\mc{J}_j)$ as the $s\times s$ submatrix obtained by extracting the rows and the columns defined by the appropriate index sets, then it is clear that  $\matc{H}_s(\mc{I}_i,\mc{J}_j)$ is a Hankel matrix, and that 
\[ \mat{y}(\mc{I}_i) = \sum_{j=1}^m \matc{H}_s(\mc{I}_i,\mc{J}_j)\mat{x}(\mc{J}_j), \qquad i = 1,\dots,\ell,\]
where $\mat{y}(\mc{I}_i)$ and $\mat{x}(\mc{J}_j)$ are the $s\times 1$ vectors obtained from $\mat{y}$ and $\mat{x}$ respectively.

Algorithm~\ref{alg:blockmult} gives the details of the procedure described here in MATLAB-like notation. It is important to note the following points. First, we need not actually form either the full block Hankel matrix or the intermediate Hankel matrices for each block element. Since the Hankel matrices are defined by the first row and the last column, we extract these quantities from the blocks $\{\mat{h}_k \}_{k=1}^{2s-1}$ as and when required. Second, since each Hankel matvec costs $\mc{O}(s\log_2 s)$, the overall cost of one matvec is $\mc{O}(\ell m s \log_2 s)$, compared to $\mc{O}(\ell m s^2)$ using the na\"\i ve approach. Finally, we can easily adapt this algorithm to compute $\matc{H}_s^\top \mat{x}$ as well; the main difference involves taking as inputs the transpose of the Markov parameters $\{\mat{h}_k^\top\}_{j=1}^{2s-1}$ instead. 

\begin{algorithm}[!ht]
\begin{algorithmic}[1]
\REQUIRE  blocks $\{\mat{h}_k\}_{k=1}^{2s-1}$ of Hankel matrix $\boldsymbol{\mc{H}} \in \mb{R}^{s\ell \times sm}$, vector $\mat{x} \in \mb{R}^{sm}$, dimension $s \geq 1$ 
\ENSURE matvec $\mat{y} = \matc{H}_s\mat{x}$
\FOR {$i = 1:\ell$}
\FOR {$j = 1:m$}
\STATE Extract first row $\mat{j}_r$ and last column $\mat{j}_c$ as $$\mat{j}_r = \bmat{\mat{h}_1(i,j) & \mat{h}_2(i,j) & \cdots & \mat{h}_s(i,j)}$$ $$\mat{j}_c = \bmat{\mat{h}_s(i,j) & \mat{h}_{s+1}(i,j) & \cdots & \mat{h}_{2s-1}(i,j)}^\top$$
\STATE Compute $\mat{\hat{y}} = $ Hankel-matvec$(\mat{j}_c,\mat{j}_r,\mat{x}(j:m:$end))
\STATE Compute $\mat{y}(i:\ell:\text{end}) =  \mat{y}(i:\ell:\text{end}) + \mat{\hat{y}}(1:s)$ 
\ENDFOR
\ENDFOR
\end{algorithmic}
\caption{$\mat{y} =$ Block-Hankel-matvec($\{\mat{h}_k\}_{k=1}^{2s-1},\mat{x},s)$}
\label{alg:blockmult}
\end{algorithm}

\subsubsection{Randomized ERA}
We now incorporate the block Hankel multiplication algorithm, Algorithm~\ref{alg:blockmult}, with RandSVD in order to accelerate the system identification process. This is simple to do; every time we need to multiply $\matc{H}_{s}$ or $\matc{H}^\top_{s}$, we use block Hankel multiplication instead. This is beneficial in two ways: to reduce the computational cost by two orders of magnitude (one from RandSVD and one from using block Hankel structure), and to reduce storage. We need not store $\matc{H}_{s}$ explicitly or even form the full matrix to begin with. All we need are the blocks that make up $\matc{H}_{s}$. This is a major benefit over the standard algorithms. 

 First, we use the RandSVD algorithm to compute a low-rank approximation of $\matc{H}_s$ with target rank $r \leq n$ to obtain $\mathat{U}_r\mathat{\Sigma}_r\mathat{V}_r^\top$. As mentioned earlier, the matvecs involving  $\matc{H}_{s}$ or $\matc{H}_{s}^\top$ are handled using Algorithm~\ref{alg:blockmult}. Then, in the system identification phase, we partition the left singular vectors as
$$\mathat{U}_r = \bmat{\mathat{\Upsilon}_f \\ *} = \bmat{* \\ \mathat{\Upsilon}_l},$$
where $\mathat{\Upsilon}_f$ and $\mathat{\Upsilon}_l$ are both $(s-1)\ell \times r$. The system matrices can then be recovered as $\mathat{A}_r = \mathat{\Upsilon}_f^\dagger \mathat{\Upsilon}_l$, 
$$\mathat{C}_r = \bmat{\mat{I}_\ell & \mat{0}}\mathat{U}_r, \qquad \mathat{B}_r = \mathat{\Sigma}_r\mathat{V}_r^\top \bmat{\mat{I}_m \\ \mat{0}}.$$
As before, the matrix $\B{D}_r$ is simply the first Markov parameter $\mat{h}_0$. We will refer to this randomized ERA that uses block Hankel multiplication as RandSVD-H. Now, we review the computational cost of this algorithm and in Section~\ref{sec:era}, we derive error bounds for the recovered system matrices.
\paragraph{Computational cost}
We now examine the computational cost of the three system identification algorithms described so far, namely the ERA using a full SVD, ERA using RandSVD, and RandSVD-H. The dominant cost of each algorithm is the SVD step, so we focus our attention there. The complexity of the SVD step for these algorithms is shown in Table~\ref{tab:cost}. Recall that the size of the matrix we are computing with is $s\ell \times sm$, the target rank for each RandSVD is $r$, and the size of the system is $n$. The cost of the full SVD is then $\mc{O}(s^3m\ell \min\{\ell,m\})$. To analyze the RandSVD based algorithms, we follow the analysis of cost in \eqref{eqn:randcost}. For a standard RandSVD, the cost of a matvec is $\mc{O}(s^2\ell m)$. If the matrix $\matc{H}_s$ is stored explicitly as we do in our implementation, then the storage cost is $\mc{O}(ms^2\ell)$; however, RandSVD can be implemented without storing $\matc{H}_s$ explicitly, in which case the cost of storage is $\mc{O}(ms\ell)$. When we use block Hankel structure to accelerate the RandSVD, the cost of a matvec is reduced to $\mc{O}(\ell m s \log_2 s)$ and the storage cost is also $\mc{O}(ms\ell)$. This shows that, as anticipated, replacing the SVD with a RandSVD reduces the cost, and then exploiting the block Hankel structure reduces the cost even further. We also include, for comparison purposes, the computational and storage costs of CUR-ERA \cite{kramer2018system}.

\begin{table}[!ht]
\centering
\begin{tabular}{c | c | c}
System ID Algorithm & Computational Cost & Storage cost \\
\hline
Full SVD & $\mc{O}(s^3m\ell \min\{\ell,m\})$ & $\mc{O}(s^2m\ell)$ \\
RandSVD  &  $\mc{O}(rs^2 \ell m +r^2s(\ell + m)$ & $\mc{O}(s^2m\ell)$ \\ 
RandSVD-H & $\mc{O}( r \ell ms\log_2s  + r^2s(\ell+m))$ & $\mc{O}(sm\ell)$\\ 
CUR-ERA & $\mc{O}(\kappa r^3+r^2sc(\ell+m))$ & $\mc{O}(sm\ell)$
\end{tabular}
\caption{Computational complexity of the dominant step, the SVD step, in each of the four algorithms.}
\label{tab:cost}
\end{table} 

\paragraph{CUR-ERA} We now include a brief description of CUR-ERA. The CUR-ERA algorithm relies on the principle of finding a maximum volume sub-matrix for computing a low-rank approximation; that is, a sub-matrix of specified dimensions with the largest determinant in absolute magnitude. Finding the maximum volume sub-matrix is a combinatorial optimization problem and, hence, one has to settle for heuristics to compute a nearly maximum volume sub-matrix in a reasonable computational time. The two approximations CUR-ERA uses are finding the cross approximation of a matrix, and finding a dominant volume submatrix. Note that in the computational cost in Table~\ref{tab:compcost}, $\kappa$ and $c$ are the number of iterations used in the cross approximation and the dominant volume submatrix, respectively. These values can be large in practice, meaning that our algorithms have a comparable computational cost. Also, the randomized SVD algorithm does not involve a combinatorial optimization problem, is known to be computationally efficient and accurate for a range of problems, and has well-developed error analysis. This makes the RandSVD-H beneficial in practical applications.  In addition, as we will show in numerical experiments in Section~\ref{sec:num}, our algorithms are more accurate.

\subsection{Randomized TERA}\label{ssec:randtera}
This approach is inspired by the tangential interpolation approach for model reduction. The goal of TERA is to reduce the dimension of the Markov parameters by projecting the parameters into a lower dimensional space. This reduces the size of the block Hankel matrix but preserves the block Hankel structure, therefore, reducing the computational cost of the SVD step~\cite{kramer2016tangential}. We briefly review the TERA approach and describe our acceleration using RandSVD.

\paragraph{TERA} In this approach, we seek two orthogonal projection matrices
\[ \begin{aligned}
    \mat{P}_1 = \mat{W}_1\mat{W}_1^\top &\qquad  \rank(\mat{W}_1) = \ell' \\
    \mat{P}_2 = \mat{W}_2\mat{W}_2^\top & \qquad \rank(\mat{W}_1) = m',
\end{aligned} \]
and the matrices $\mat{W}_1$ and $\mat{W}_2$ have orthonormal columns. To compute $\mat{P}_1$, we first arrange the Markov parameters in the matrix 
\begin{equation}\label{eqn:Hw} \matc{H}_w = \bmat{\mat{h}_1 & \dots & \mat{h}_{2s-1}} \in \mb{R}^{\ell \times m(2s-1)}.\end{equation}
We then solve the optimization problem $$\min_{\rank(\mat{P}) = \ell'} \|\mat{P}\matc{H}_w - \matc{H}_w \|_F^2 .$$ The optimal solution can be constructed using the SVD of $\matc{H}_w = \mat{U}_w\mat\Sigma_w\mat{V}_w^\top$. We take $\mat{W}_1 = \mat{U}_w(:,1:\ell')$; that is, we take $\mat{W}_1$ to be the first $\ell'$ left singular vectors of $\matc{H}_w$. Similarly, to construct the matrix $\mat{W}_2$, we first arrange the Markov parameters into the matrix 
\begin{equation}\label{eqn:He}\matc{H}_e = \bmat{\mat{h}_1 \\ \vdots \\ \mat{h}_{2s-1}} \in \mb{R}^{\ell (2s-1) \times m}. \end{equation}
We then solve the optimization problem $$\min_{\rank(\mat{P}) = m'} \|\matc{H}_e - \matc{H}_e\mat{P} \|_F^2 .$$ The optimal solution can be constructed using the SVD of $\matc{H}_e = \mat{U}_e\mat\Sigma_e\mat{V}_e^\top$. We take $\mat{W}_2 = \mat{V}_e(:,1:m')$; that is, we take $\mat{W}_2$ to be the first $m'$ right singular vectors of $\matc{H}_e$. Having obtained the matrices $\mat{W}_1$ and $\mat{W}_2$, we construct the projected Markov parameters as 
\[ \widetilde{\mat{h}}_i = \mat{W}_1^\top \mat{h}_i \mat{W}_2  \in \mb{R}^{\ell' \times m'}, \qquad i = 1,\dots,2s-1. \]
The block Hankel matrix $\matc{H}_{s}$ is ``projected'' using the matrices $\mat{W}_1$ and $\mat{W}_2$ arranged in diagonal blocks to obtain the reduced-size block Hankel matrix 
\begin{equation}
    \widetilde{\matc{H}}_{s} =  \bmat{\B{W}_1^\top \\ & \B{W}_1^\top \\ & & \ddots \\ & && \B{W}_1^\top  }\bmat{\mat{h}_1 & \mat{h}_2 & \dots & \mat{h}_s\\ \mat{h}_2 & \mat{h}_3 & \dots & \mat{h}_{s+1} \\ \vdots & \vdots & \iddots & \vdots \\ \mat{h}_s & \mat{h}_{s+1} & \dots & \mat{h}_{2s-1}} \bmat{\B{W}_2 \\ & \B{W}_2\\ & & \ddots \\ & && \B{W}_2  } .
    \end{equation}
It is important to observe that due to this projection, the block Hankel structure is preserved, and we can express $\widetilde{\matc{H}}_{s}$ in terms of the ``projected'' Markov parameters as 
 \begin{equation}
  \widetilde{\matc{H}}_{s}  = \bmat{\mattilde{h}_1 & \mattilde{h}_2 & \dots & \mattilde{h}_s   \\ \mattilde{h}_2 & \mattilde{h}_3 & \dots & \mattilde{h}_{s+1} \\ \vdots & \vdots & \iddots & \vdots \\ \mattilde{h}_s & \mattilde{h}_{s+1} & \dots & \mattilde{h}_{2s-1}} \in \mb{R}^{(\ell's)\times (m't)}.
\end{equation}
If $\ell' \ll \ell$ and $m' \ll m$, then the size of the matrix $\widetilde{\matc{H}}_{s}$ is much less than $\matc{H}_{s}$. The dimensions $\ell'$ and $m'$ are determined by the singular value decay of the matrices $\matc{H}_w$ and $\matc{H}_e$. Retaining a larger number of singular vectors (that is, large $\ell'$ and $m'$) results in a more accurate approximation to $\matc{H}_{s}$ but results in a larger matrix $\widetilde{\matc{H}}_{s}$ and in a higher computational cost.

\paragraph{Recovering system matrices} The next steps mimic the standard ERA approach to reconstruct the system matrices but we must carefully account for the dimensions of the projected system. We compute the SVD of $\widetilde{\matc{H}}_{s} = \mattilde{U}_r\mattilde{\Sigma}_r\mattilde{V}_r^\top$, and partition the left singular vectors as 
\[\mattilde{U}_r=\bmat{\mattilde{\Upsilon}_f \\ \mat{*}} = \bmat{\mat{*} \\ \mattilde{\Upsilon}_l}.\]
The system matrices can be recovered as 
\begin{equation} \label{eqn:tera} \mattilde{A}_r = \>\mattilde{\Upsilon}_f^\dagger \mattilde{\Upsilon}_l \in \mb{R}^{r \times r}, \qquad  \mattilde{C}_r = \>  \bmat{\mat{W}_1 & \mat{0}} \mattilde{\Upsilon}_f, \qquad   \mattilde{B}_r = \mattilde{\Sigma}_r\mattilde{V}_r^\top\bmat{\mat{W}_2^\top\\ \mat{0}} .  \end{equation}
\paragraph{Computational cost} To motivate the need for accelerating TERA using RandSVD, we first review the computational cost which has three components. Computing the projection matrices $\mat{W}_1$ and $\mat{W}_2$ involves a computational cost of $\mc{O}\left(s(\ell m^2 + m\ell^2) \right)$. The SVD of $\widetilde{\matc{H}}_{s}$ now costs 
\[\mc{O}( s^3\ell'm' \min\{\ell',m'\}),\]
and, similar to earlier algorithms, the cost recovery of the system matrices is negligible compared to the cost of the SVD. Although the computational cost is significantly reduced (assuming $\ell' \ll \ell$ and $m' \ll m$), the cost of the SVD still dominates the computational cost and has cubic scaling with $s$. We now propose a new algorithm to lower this cost. 

\begin{algorithm}[!ht]
\begin{algorithmic}[1]
\REQUIRE  blocks $\{\mat{h}_k\}_{k=1}^{2s-1}$ of Hankel matrix $\boldsymbol{\mc{H}} \in \mb{R}^{s\ell \times sm}$, target rank $1\leq r \leq n$, integers $\ell' \leq \ell$, $m'\leq m$.
\STATE Form $\matc{H}_w$ using~\eqref{eqn:Hw} and compute its SVD $\mat{U}_w\mat\Sigma_w \mat{V}_w^\top$. Set $\mat{W}_1 = \mat{U}_w(:,1:\ell')$
\STATE Form $\matc{H}_e$ using~\eqref{eqn:He} and compute its SVD $\mat{U}_e\mat\Sigma_e \mat{V}_e^\top$. Set $\mat{W}_2 = \mat{V}_e(:,1:m')$
\STATE Form $\mathat{h}_i = \mat{W}_1^\top\mat{h}_i\mat{W}_2$ for $i=1,\dots,2s-1$.
\STATE Compute $[\mattilde{U}_r,\mattilde{\Sigma}_r,\mattilde{V}_r]$=RandSVD$(\widetilde{\matc{H}}_s,r,\rho)$. \COMMENT{Matvecs are computed using Algorithm~\ref{alg:blockmult}}
\STATE Compute system matrices using \eqref{eqn:tera}
\end{algorithmic}
\caption{ RandTERA}
\label{alg:randtera}
\end{algorithm}

\paragraph{RandTERA} It is worth pointing out that the ``projected'' Hankel matrix $\widetilde{\matc{H}}_{s}$ still retains its block Hankel structure. To reduce the computational cost, we combine the following ingredients that were used previously: we use the RandSVD to compute the rank$-r$ approximation to $\widetilde{\matc{H}}_{s}$, the matvecs involving $\widetilde{\matc{H}}_{s}$ can be accelerated using the block Hankel structure (as described in Section~\ref{sssec:blockHankel}). The main difference is that we use the ``projected'' Markov parameters $\{\mattilde{h}_i\}_{i=1}^{2s-1}$ rather than the Markov parameters $\{\mat{h}_i\}_{i=1}^{2s-1}$. As a result, the computational cost of the SVD step is reduced to 
\[\mc{O}\left( r\ell'm' s\log_2s  + r^2s(m' + \ell') \right).\]
We call this algorithm RandTERA, and a complete description of this algorithm is given in Algorithm~\ref{alg:randtera}. A breakdown of the computational cost of the RandTERA is given below in Table~\ref{tab:compcost}.

\begin{table}[!ht]
\begin{center}
\begin{tabular}{c | c }
Stage & Computational Cost  \\
\hline
Computing $\B{W}_1,\B{W}_2$ & $\mc{O}\left(s(\ell m^2 + m\ell^2) \right)$\\
SVD step & $\mc{O}\left( r\ell'm' s\log_2s  + r^2s(m' + \ell') \right)$  
\end{tabular}
\end{center}
\caption{Computational cost of computing RandTERA}
\label{tab:compcost}
\end{table}
Regarding the storage cost, since only the projected Markov parameters need to be stored, the storage cost is $\mc{O}(sm'\ell')$, which is potentially lower than RandSVD-H which requires $\mc{O}(sm\ell)$.

\section{Error Analysis}\label{sec:err}

In this section, we analyze the accuracy of RandERA in Section~\ref{sec:randera}, and derive bounds on the accuracy of the recovered system matrix $\mat{A}_r$. Our derivation is not tied to the randomized algorithms used to compute the approximations; this makes the analysis applicable to more general settings. 

\subsection{Background and assumptions}\label{ssec:assump} We first review the necessary background information and clearly state our assumptions needed for the analysis. 
\paragraph{Canonical Angles} Given $\matc{H}_s = \mat{U\Sigma V}^\top$, partition $\mat{U}_r = \mat{U}(:,1:r)$, the left singular vectors corresponding to the top $r$ singular values of $\matc{H}_s$, as
\[ \B{U}_r = \bmat{\B{\Upsilon}_f\\ \mat{*}} = \bmat{\mat{*} \\\B{\Upsilon}_l} \in \mb{R}^{s\ell \times r}.\]
Note that we have dropped the superscripts $(r)$, compared to Section~\ref{sec:era}, to make the notation manageable. We can recover the matrix $\B{A}_r$ as  $\B{A}_r = \B{\Upsilon}_f^\dagger\B{\Upsilon}_l$. Similarly, let $\Hh_s = \mathat{U}\mathat{\Sigma}\mathat{V}^\top$ be an approximation to $\B{\mc{H}}_s$ with left singular vectors $\mathat{U}_r$ partitioned as
\[ \Uh_r = \bmat{\mathat{\Upsilon}_f\\ \mat{*}} = \bmat{\mat{*} \\\mathat{\Upsilon}_l} \in \mb{R}^{s\ell \times r}, \]
and we can compute the approximate system matrix $\mathat{A}_r = \mathat{\Upsilon}_f^\dagger \mathat{\Upsilon}_l \in \mb{R}^{r\times r}$. Let the canonical angles between the subspaces $\range (\B{U}_r)$ and $\range (\mathat{U}_r)$ be denoted by
\[ 0 \leq \theta_1 \leq \dots \leq \theta_{max} \leq \pi/2. \]
If we collect the angles into a diagonal matrix $\B\Theta$, then  $\B{U}_r^\top\Uh_r$ has the SVD 
\[ \B{U}^\top_r\Uh_r = \B{P}(\cos\B\Theta)\B{Q}^\top.\]
Let $\proj{\B{U}_r}$ denote the orthogonal projector onto $\range(\B{U}_r)$; similarly, let $\proj{\Uh_r}$ denote the orthogonal projector onto $\range(\Uh_r)$. Then, the distance between the two subspace $\range(\B{U}_r)$ and $\range(\Uh_r)$ is given by 
\[\|\proj{\B{U}_r} - \proj{\Uh_r}\|_2 = \|(\B{I}-\proj{\B{U}_r})\proj{\Uh_r}\|_2 = \|\sin\B\Theta\|_2 =  \sin\theta_{\max}. \]
See \cite[Chapter II.4]{stewart1990matrix} for more details on canonical angles between subspaces.

\paragraph{Pseudoinverses} We recall some facts about the perturbation of pseudoinverses~\cite[Section 2.2.2]{bjorck2015numerical}. If $\B{M},\B{E} \in \mb{R}^{n\times r}$ such that $\rank(\B{M}) = \rank(\B{M}+\B{E}) = r$ then
\begin{equation}\label{eqn:mppert}
    \|(\B{M} + \B{E})^\dagger - \B{M}^\dagger\|_2 \leq \sqrt{2}\|\B{M}^\dagger\|_2\| (\B{M}+\B{E})^\dagger\|_2 \|\B{E}\|_2
\end{equation}
Furthermore, if $\|\B{E}\|_2 \|\B{M}^\dagger\|_2 < 1$ then
\begin{equation}\label{eqn:mpinv} \| (\B{M}+\B{E})^\dagger\|_2 \leq  \frac{\|\B{M}^\dagger\|_2}{1-\|\B{E}\|_2\|\B{M}^\dagger\|_2}. \end{equation}
If $\B{A}\in \mb{R}^{m\times r}$ and $\B{B}\in \mb{R}^{r\times n}$ such that $\rank(\B{A}) = \rank(\B{B}) = r$, then by~\cite[Theorem 2.2.3]{bjorck2015numerical}
\begin{equation}\label{eqn:abdagger} (\B{AB})^\dagger = \B{B}^\dagger \B{A}^\dagger.\end{equation}

\paragraph{Accuracy of eigenvalues} A norm based approach, i.e., $\|\mat{A}-\Ah_r\|$, is not meaningful since  $\mat{A}_r$ (and $\Ah_r$) can only be determined up to a similarity transformation. To compare the approximate system matrix $\Ah_r$ with $\mat{A}_r$, we compare the eigenvalues of these two matrices since the eigenvalues remain unchanged by a similarity transformation.  We measure the accuracy of the eigenvalues of $\Ah$ using the spectral variation, which we now define. Let $\B{A},\B{B}\in\mb{R}^{n\times n}$; the spectral variation between $\B{A}$ and $\B{B}$ is defined as~\cite[Section VI.3]{bhatia2013matrix}
\begin{equation}\label{eqn:spvar}
\sv(\psi(\B{B}),\psi(\B{A})) = \max_{1 \leq j \leq n} \min_{ 1\leq i \leq n} |\lambda_i(\B{A}) -\lambda_j(\B{B})|,
\end{equation}
where $\psi(\cdot)$ denotes the spectrum of the matrix. 

We briefly list the various assumptions that are required for our main result.
\begin{assumption}\label{ass:main}  We assume:
\begin{enumerate}
\item  [A1.] System matrix:  Assume that $\B{A}_r = \B{\Upsilon}_f^\dagger\B{\Upsilon}_l \in \mb{R}^{r\times r}$ is diagonalizable and let $\B{W} \in \mb{R}^{r\times r}$ be the matrix of eigenvectors.
\item [A2.] Singular vectors: Assume that the subblocks $\B{\Upsilon}_f$ of the singular vectors $\B{U}_r$ are full rank; that is  
$\rank(\B{\Upsilon}_f) =  r$. 
\item [A3.]
Markov parameters: Assume the Markov parameters converge to $\B{h}_i \rightarrow \B{0}$ for $i > s$. 
\item  [A4.] Canonical angles: Assume that the canonical angles are sufficiently small and satisfy
\begin{equation}\label{eqn:eta}
    \eta \equiv 2\sin\theta_{\max} \|\B{\Upsilon}_f^\dagger\|_2 < 1.
\end{equation} 
\end{enumerate}

\end{assumption}

Assumption A1 is rather strong but can be weakened, if different perturbation results are used; see~\cite[Chapter VIII. 1]{bhatia2013matrix}. Assumption A2 ensures that the least squares problem $\min_{\B{A}}\| \B{\Upsilon}_f\B{A} - \B{\Upsilon}_l\|_F$  
has a unique solution. Assumption A3 is also made in~\cite{kung1978new}, and is necessary to ensure the stability of the system. Finally, Assumption A4 ensures that the approximate singular vectors are sufficiently accurate. 

\subsection{Main result} We are ready to state our main theorem. 
\begin{theorem}\label{thm:error} With the notation in Section~\ref{ssec:assump} and \cref{ass:main} 
\[ \sv(\psi(\Ah_r), \psi(\B{A}_r)) \leq \kappa_2(\B{W}) \eta\left( 1 +\frac{ \sqrt{2}\|\B{\Upsilon}_f^\dagger\|_2}{1-\eta} \right),  \]
where $\kappa_2(\B{W}) = \|\B{W}\|_2\|\B{W}^{-1}\|_2$ is the condition number of the eigenvectors $\B{W}$ of $\mat{A}_r$.
\end{theorem}
The theorem identifies several factors that can cause large errors during the identification step. First,  the eigenvectors of $\B{W}$ may be ill-conditioned, i.e., $\kappa_2(\B{W})$ can be large. The best case scenario is when $\B{A}$ is normal, so that the condition number is $1$.  Second, the norm of the pseudoinverse $ \|\B{\Upsilon}_f^\dagger\|_2 \geq 1$ can be large. However, Assumption A3 ensures that $\lim_{s\rightarrow \infty} \|\B\Upsilon_f\|_2 =1$; see~\cite[Lemma 5]{kramer2018system} for a detailed argument.  Finally, if the singular vectors of the low-rank approximation are not computed accurately, then the canonical angles can be large, i.e., $\sin\theta_{\max}$ can be close to $1$. This can be mitigated, for example, by taking several subspace iterations. The first three assumptions are intrinsic to the system $(\B{A},\B{B},\B{C},\B{D})$ and only Assumption A4 depends on the choice of the numerical method. The theorem is applicable to the randomized algorithms developed here, namely, RandERA and RandTERA. However, it is important in a larger context, since it is applicable to any approximation algorithm for ERA, including TERA.

\begin{proof}[Proof of Theorem~\ref{thm:error}]
There are several steps involved in this proof. 

\textbf{1. Introducing a similarity transformation.}  Let $\B{Z}\in \mb{R}^{r\times r}$ be an orthogonal matrix; we will leave the specific choice of this matrix to step 2. Since $\B{A}_r$ is diagonalizable, using a perturbation theorem for diagonalizable matrices~\cite[Theorem VIII.3.1]{bhatia2013matrix}, we have
\begin{equation}\label{eqn:perturb1} \sv(\psi(\Ah_r), \psi (\B{A}_r)) = \sv(\psi(\B{Z}\Ah_r\B{Z}^\top), \psi(\B{A})) \leq \kappa_2(\B{W}) \|\B{A}_r - \B{Z}\Ah_r\B{Z}^\top\|_2. \end{equation}
 The equality in the above equation is because the eigenvalues of $\Ah_r$ and $\B{Z}\Ah_r\B{Z}^\top$ are the same. We first discuss the choice of $\B{Z}$ before analyzing the error in the term $\|\B{A}_r - \B{Z}\Ah_r\B{Z}^\top\|_2$.

\textbf{2. Choosing $\B{Z}$.} Recall that the SVD of $\B{U}_r^\top\Uh_r$ is $\B{P}(\cos\B\Theta)\B{Q}^\top$. Using the unitary invariance of the 2-norm, $\|\B{U}_r-\Uh_r\B{Z}^\top \|_2 = \| \B{U}_r\B{Z} - \Uh_r \|_2$. Now, we choose $\B{Z} = \B{PQ}^\top$ and verify that $\B{Z}$ is orthogonal. By the triangle inequality and the unitary invariance of the 2-norm,
\begin{equation}\label{eqn:umuz}\begin{aligned}\| \B{U}_r\B{Z} - \Uh_r \|_2 \leq & \> \|\B{U}_r\B{Z} - \B{U}_r\B{U}_r^\top\Uh_r \|_2 + \|(\B{I}-\B{U}_r\B{U}_r^\top)\Uh_r \|_2 \\
\leq & \> \|\B{U}_r\B{Z} - \B{U}_r\B{P}(\cos\B\Theta)\B{Q}^\top \|_2 + \|(\B{I}-\B{U}_r\B{U}_r^\top)\proj{\Uh_r}\Uh_r\|_2\\
\leq & \> \|\B{PQ}^\top - \B{P}(\cos\B\Theta)\B{Q}^\top \|_2 + \|(\B{I}-\proj{\B{U}_r})\proj{\Uh_r}\|_2 \|\Uh_r\|_2\\
\leq & \>\|\B{I} - \cos\B\Theta \|_2 + \|\sin\B\Theta \|_2.
\end{aligned}\end{equation}
Since the canonical angles satisfy $0 \leq \theta_i \leq \pi/2$, we have $\cos^2\theta_i \leq \cos\theta_i$ and, therefore, for $i=1,\dots,r$
\[ (1-\cos\theta_i)^2 = 1 -2\cos\theta_i + \cos^2\theta_i \leq 1-\cos^2\theta_i = \sin^2\theta_i.\]
Therefore, $\|\B{I} - \cos\B\Theta \|_2 \leq \|\sin\B\Theta\|_2$. Together with \eqref{eqn:umuz}, we have 
\begin{equation} \label{eqn:perturb2} \|\B{U}_r-\Uh_r\B{Z}^\top \|_2=\| \B{U}_r\B{Z} - \Uh_r \|_2 \leq 2 \|\sin\B\Theta\|_2.\end{equation}

\textbf{3. Ensuring  $\rank(\mathat{\Upsilon}_f)=r$.} 
Choose $\B{M} = \B{\Upsilon}_f$, which has rank $r$, and $\B{M}+\B{E} = \mathat{\Upsilon}_f\B{Z}^\top$. We now show that $\B{M}+\B{E}$ also has rank $r$. 
Using Weyl's theorem on perturbation of singular values~\cite[Theorem 2.2.8]{bjorck2015numerical},
\[ \begin{aligned}|\sigma_r(\mathat{\Upsilon}_f\B{Z}^\top) - \sigma_r(\B\Upsilon_f)| \leq & \> \|\B{E}\|_2 = \|\B{\Upsilon}_f-\mathat{\Upsilon}_f\B{Z}^\top\|_2 \leq  \|\B{U}_r-\Uh_r\B{Z}^\top\|_2 \leq 2 \|\sin\B\Theta\|_2 .
\end{aligned} \]
We have used the fact that $\B{\Upsilon}_f$ and $\mathat{\Upsilon}_f\B{Z}^\top$ are submatrices of $\B{U}_r$ and $\Uh_r\B{Z}^\top$, respectively. Since $\sigma_r(\B\Upsilon_f) = 1/\|\B\Upsilon_f^\dagger\|_2$, we can rearrange to get
\[ \sigma_r(\mathat{\Upsilon}_f\B{Z}^\top) \geq  \frac{1}{\|\B\Upsilon_f^\dagger\|_2} - 2\|\sin\B\Theta\|_2  > 0\] 
since by assumption, $\eta = 2 \|\sin\B\Theta\|_2\|\B{\Upsilon}_f^\dagger\|_2 < 1$. This shows that both $\mathat{\Upsilon}_f$ and $\mathat{\Upsilon}_f\B{Z}^\top$ have rank $r$. Therefore, by~\eqref{eqn:abdagger}, $(\mathat{\Upsilon}_f\B{Z}^\top)^\dagger = \B{Z}\mathat{\Upsilon}_f^\dagger$. Furthermore,~\eqref{eqn:mpinv} applies and
\begin{equation}\label{eqn:perturb3}
    \|\B{Z}\mathat{\Upsilon}_f^\dagger\|_2 \leq \frac{\|\B{\Upsilon}_f^\dagger\|_2}{1-2 \|\sin\B\Theta\|_2\|\B{\Upsilon}_f^\dagger\|_2}. 
\end{equation}

\textbf{4. Error in $\|\B{A}_r - \B{Z}\Ah_r\B{Z}^\top\|_2$.} Write both matrices in terms of their factors, and add and subtract the term $\B{\Upsilon}_f^\dagger \mathat{\Upsilon}_l\B{Z}^\top$ to get
\begin{equation}\label{eqn:perturb}
\begin{aligned}
\| \B{A}_r -  \B{Z}\Ah_r\B{Z}^\top\|_2 &= \| \B{\Upsilon}^\dagger_f\B{\Upsilon}_l - \B{\Upsilon}_f^\dagger\mathat{\Upsilon}_l\B{Z}^\top + \B{\Upsilon}_f^\dagger\mathat{\Upsilon}_l\B{Z}^\top -\B{Z}\mathat{\Upsilon}_f^\dagger\mathat{\Upsilon}_l\B{Z}^\top \|_2 \\
&\leq \| \B{\Upsilon}^\dagger_f \|_2\|\B{\Upsilon}_l-\mathat{\Upsilon}_l\B{Z}^\top \|_2 + \|\B{\Upsilon}^\dagger - \B{Z}\mathat{\Upsilon}_f^\dagger \|_2\| \mathat{\Upsilon}_l\B{Z}^\top \|_2.
\end{aligned}
\end{equation}
Note that, by Assumption A2, $\rank(\B{\Upsilon}_f) = r$. In step 3, we showed that $\rank(\mathat{\Upsilon}_f\B{Z}^\top) = r$.  
As before, with $\B{M} = \B{\Upsilon}_f$ and $\B{M}+\B{E} = \mathat{\Upsilon}_f\B{Z}^\top$,~\eqref{eqn:mppert} applies and
$$\|\B{\Upsilon}_f^\dagger - \B{Z}\mathat{\Upsilon}_f^\dagger \|_2 \leq  \sqrt{2} \|\B{\Upsilon}_f^\dagger \|_2 \|\B{Z}\mathat{\Upsilon}_f^\dagger\|_2 \| \B{\Upsilon}_f - \mathat{\Upsilon}_f\B{Z}^\top \|_2.$$
Plugging this into~\eqref{eqn:perturb}, we have
\[\| \B{A}_r -  \B{Z}\Ah_r\B{Z}^\top\|_2 \leq   \|\B{\Upsilon}_f^\dagger \|_2 \left(\|\B{\Upsilon}_l-\mathat{\Upsilon}_l\B{Z}^\top \|_2  +  \sqrt{2}\|\B{Z}\mathat{\Upsilon}_f^\dagger\|_2 \| \B{\Upsilon}_f - \mathat{\Upsilon}_f\B{Z}^\top \|_2\right).\]
We have used the fact that $\| \mathat{\Upsilon}_l\B{Z}^\top \|_2 \leq \|\Uh_r\B{Z}^\top\|_2 = 1$; the inequality is because $\mathat{\Upsilon}_l\B{Z}^\top$ is a submatrix of $\Uh_r\B{Z}^\top$ and the equality follows since $\Uh_r\B{Z}^\top$ has orthonormal columns. Similarly,
\[ \|\B{\Upsilon}_l-\mathat{\Upsilon}_l\B{Z}^\top\|_2 \leq \|\B{U}_r-\Uh_r\B{Z}^\top\|_2, \qquad \|\B{\Upsilon}_f-\mathat{\Upsilon}_f\B{Z}^\top\|_2 \leq \|\B{U}_r-\Uh_r\B{Z}^\top\|_2,\]
so that 
\begin{equation}\label{eqn:perturb4}\| \B{A}_r -  \B{Z}\Ah_r\B{Z}^\top\|_2 \leq   \|\B{\Upsilon}_f^\dagger \|_2 \|\B{U}_r-\Uh_r\B{Z}^\top \|_2\left( 1  +  \sqrt{2}\|\B{Z}\mathat{\Upsilon}_f^\dagger\|_2 \| \right).\end{equation}

\textbf{5. Finishing the proof.} Plug~\eqref{eqn:perturb2} and~\eqref{eqn:perturb3} into~\eqref{eqn:perturb4} to obtain
\[ \| \B{A}_r -  \B{Z}\Ah_r\B{Z}^\top\|_2  \leq  (2\|\sin\B\Theta\|_2 \|\B{\Upsilon}_f^\dagger\|_2)\left(1 +  \frac{\sqrt{2}\|\B{\Upsilon}_f^\dagger\|_2}{1-2 \|\sin\B\Theta\|_2\|\B{\Upsilon}_f^\dagger\|_2} \right).  \]
Plug this bound into~\eqref{eqn:perturb1} to finish the proof. 
\end{proof}

An important aspect of the proof of Theorem~\ref{thm:error} is the choice of the orthogonal matrix $\B{Z}$ that determines the similarity transformation. Our proof used this idea from~\cite[Theorem 5]{hunter2010performance}. It is worth pointing out that the matrix $\B{Z} = \B{PQ}^\top$ is also the solution to the orthogonal Procrustes problem 
\[ \min_{\substack{\B{Z} \in \mb{R}^{r\times r} \\\B{Z}^\top\B{Z} = \B{I}_r }} \|\B{U}_r\B{Z} - \Uh_r\|_F.\]
Therefore, the matrix $\B{Z} = \B{PQ}^\top$ is the matrix that ``best rotates'' the columns $\B{U}_r$ to align with the columns of $\Uh_r$. Note, however, that in the proof we use the 2-norm instead of the Frobenius norm.

\subsection{Accuracy of the singular vectors} Theorem~\ref{thm:error} shows that the accuracy of the approximation algorithms to the ERA depend on the canonical angles between the exact and approximate subspaces $\range(\B{U}_r)$ and $\range(\Uh_r)$. Insight into the accuracy between these subspaces can be obtained by standard perturbation theory. Suppose there are numbers $\alpha \geq 0$ and $\delta >0$ such that 
\[ \sigma_{r+1}(\Hh_s) \geq \alpha + \delta, \qquad \sigma_r(\mat{H}_s) \leq \alpha. \]
Let $\Uh_r,\Vh_r$ be the matrices containing the left and right singular vectors corresponding to the first $r$ singular values of $\Hh_s$, which are taken to be the diagonals of the matrix $\Sh_r$. Then, by~\cite[Chapter V. 4, Theorem 4.4]{stewart1990matrix},
\[  \|\sin\B\Theta\|_2 \leq \frac{\max\{ \|\matc{H}_s\Vh_r- \Uh_r\Sh_r \|_2, \|\matc{H}_s^\top \Uh_r- \Vh_r \Sh_r\|_2\}}{\delta}.\]
In Section~\ref{sec:randera}, the randomized subspace iteration is used to construct the low-rank approximation $\Hh_s$. We can use more special purpose error bounds to quantify the accuracy of the singular vectors. Specifically, one may use techniques developed in Section 3.2 of~\cite{saibaba2019randomized}. We omit a detailed statement of the results.

\subsection{Stability}
The discrete dynamical system with the system matrix $\B{A}_r$ is stable if the eigenvalues of $\B{A}_r$ lie inside the unit circle; that is, the spectral radius $\rho(\B{A}_r) < 1$. Assumption A3 in Section~\ref{ssec:assump} means that the Markov parameters decay after some finite time. Then, by~\cite[Theorem 3]{kramer2018system}, there exists a positive integer $s$ such that the identified reduced order model is a stable discrete dynamical system. Theorem~\ref{thm:error} can also be used to determine when the approximate discrete dynamical system with the system matrix $\Ah_r$ is stable. 

We consider the spectral radius of $\Ah_r$, and write
\[ \rho(\Ah_r) = \max_{1 \leq j \leq n}|\lambda_j(\Ah_r)| = |\lambda_{j^*}(\Ah_r)|, \]
for some index $j^*$. Similarly, let $i^*$ be the index such that $\lambda_{i^*}(\B{A}_r)$ is the closest eigenvalue to $\lambda_{j^*}(\Ah_r)$. If $i^*$ and $j^*$ are not unique, break the tie in some fashion. Therefore, we have the series of inequalities
\[ \rho(\Ah_r) \leq |\lambda_{j^*}(\Ah_r) - \lambda_{i^*}(\B{A}_r) | + | \lambda_{i^*}(\B{A}_r)| \leq  \sv(\Psi(\Ah_r),\Psi(\B{A}_r)) + \rho(\B{A}_r). \]
By Theorem~\ref{thm:error}, the approximate discrete dynamical system is stable if 
\[  \kappa_2(\B{W}) \eta\left( 1 +\frac{ \sqrt{2}\|\B{\Upsilon}_f^\dagger\|_2}{1-\eta} \right) < 1 - \rho(\B{A}_r).\]
By assumption, $\eta < 1$; the above equation gives an additional condition on $\eta$ for stability. This condition can be informative for the numerical method used to compute the singular vectors $\Uh_r$. 

\section{Numerical Results}\label{sec:num}

To test our algorithms, we consider two different test problems which pose distinct challenges. The first test problem is generated from a LTI model of a controlled heat transfer process for optimal cooling of a steel profile \cite{morwiki_steel}. The second test problem is generated from a LTI state-space model of an electrical power generation system with 50 generators \cite{BenS05b}. The dimensions of the variables related to both test problems are shown in Table~\ref{tab:datasets}. The heat transfer test problem has a very large system size $n$ but a small number of inputs and outputs, while the power system test problem has a relatively small system size but a much larger number of inputs and outputs.  
All our experiments were run in \textsc{MATLAB}, and our timing experiments were run on the NCSU HPC Cluster with 72GB memory.

\begin{table}[!ht]
    \centering
    \begin{tabular}{c|c|c|c}
        System & System Size {$n$}  & Inputs {$m$} & Outputs {$\ell$}  \\
        \hline
         Heat Transfer & 1357 & 7 & 6 \\
         Power System & 155 & 50 & 155
    \end{tabular}
    \caption{Details of Test Problems}
    \label{tab:datasets}
\end{table}

\subsection{Heat Transfer}
We first test our algorithms on the heat transfer test problem, which has a large system size but a relatively small number of inputs and outputs. This system also has slow dynamics, so a large value of $s$ is needed to capture the transient behavior. The model for this test problem is not in standard form, so we will need to convert it to standard form before we are able to use our algorithms. The original continuous form after spatial discretization is
\begin{equation*}
    \begin{aligned}
        \mat{E} \frac{d\mat{x}}{d t} &= \mat{A}_o\mat{x} + \mat{B}_o\mat{u} \\
        \mat{y} &= \mat{C}_o\mat{x} + \mat{D}_o \mat{u}.
    \end{aligned}
\end{equation*}

To convert it to standard form, we use the Cholesky factorization of $\mat{E} = \mat{LL}^\top$. Then, the continuous standard form system matrices are found by computing $\mat{A}_c = \mat{L}^{-1}\mat{A}_o\mat{L}^{-\top}$, $\mat{B}_c = \mat{L}^{-1}\mat{B}_o$, and $\mat{C}_c = \mat{C}_o\mat{L}^{-\top}$. Note that $\mat{D}_c = \mat{D}_o$ remains unchanged. The continuous-time matrices are converted into appropriate discrete-time matrices using \textsc{MATLAB}'s \verb|c2d| command. It is worth noting that the continuous-time matrices are only used in the construction of the Markov parameters but are not used explicitly in the system identification algorithms.

For this heat transfer problem, we will compare accuracy and speed of our algorithms that take advantage of the block Hankel structure of $\matc{H}_s$, i.e., RandSVD-H (see Section~\ref{sec:randera}), to the standard ERA that uses a full SVD. We also contrast these two algorithms with another algorithm we call SVDS-H, which uses the \textsc{MATLAB} command \verb|svds|, but we instead use the block Hankel multiplication described in Algorithm~\ref{alg:blockmult}.

\subsubsection{Accuracy} We consider several numerical experiments to test the accuracy of our algorithms. For each experiment to test accuracy, we used $s = 1000$ and a target rank of $r=20$. That is, we are performing model reduction in addition to the system identification. For RandSVD-H, we used $q=1$ for the number of subspace iterations and $\rho = 20$ for the oversampling parameter.  In our first experiment, we compare the first 20 singular values of the block Hankel matrix $\matc{H}_s \in \mb{R}^{6000 \times 7000}$ as computed by a full SVD to the first $20$ singular values computed by the SVDS-H and RandSVD-H algorithms.  The singular values are plotted in Figure~\ref{fig:svalues_steel1357}; we see that they are in good agreement with each other.

\begin{figure}[!ht]
    \centering
    \includegraphics[scale=.4]{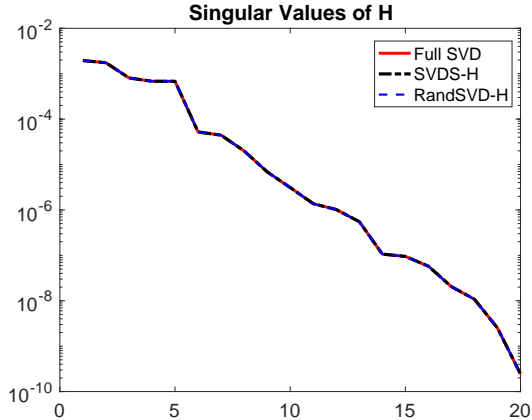}
    \caption{Singular values of the block Hankel matrix $\matc{H}_s$ as computed by a full SVD, SVDS-H, and RandSVD-H. We computed the first $20$ singular values using SVDS-H and RandSVD-H and plotted them against the first 20 singular values computed by the full SVD.}
    \label{fig:svalues_steel1357}
\end{figure}

Next we compare the eigenvalues of the identified $\mathat{A}_r$ matrices using the SVDS-H and RandSVD-H algorithms to the eigenvalues of $\mat{A}_r$ identified using the full SVD ERA algorithm.  The results are shown in Figure~\ref{fig:evalues_steel1357}, where we can see that the eigenvalues of $\mathat{A}_r$ computed using SVDS-H and RandSVD-H match those of $\mat{A}_r$ computed using the full SVD  accurately. 

\begin{figure}[!ht]
    \centering
    \includegraphics[width=\textwidth]{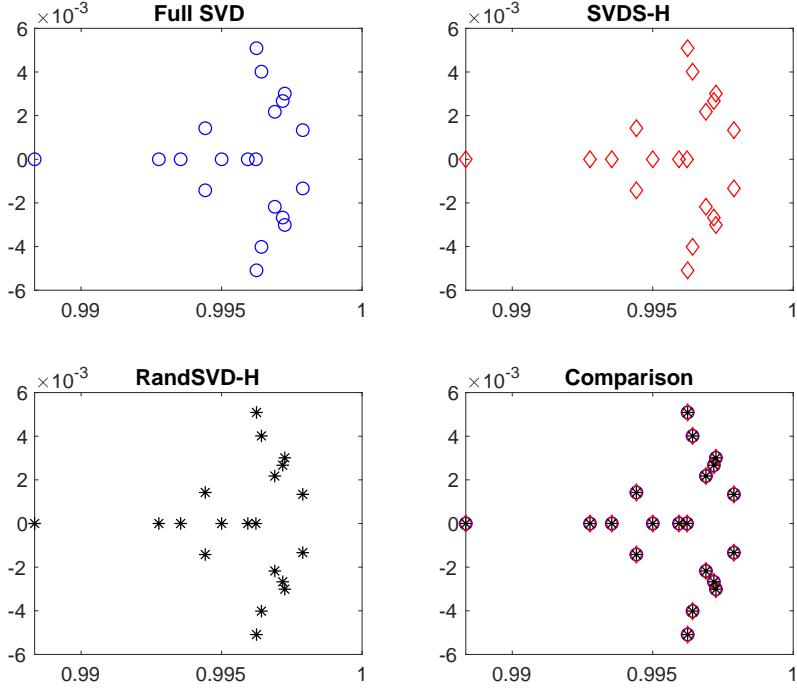}
    \caption{Eigenvalues of the identified $\mat{A}_r$ matrix using SVD, SVDS-H, and RandSVD-H algorithms on the heat transfer problem. We used $s = 1000$ and a target system size of $r = 20$ as inputs for each algorithm.}
    \label{fig:evalues_steel1357}
\end{figure}

In order to quantitatively evaluate the accuracy of the identified eigenvalues, we consider the Hausdorff distance metric, defined between the spectra of the two matrices $\mat{A}_r$ and $\mathat{A}_r$, 

\begin{equation}\label{eqn:hausdist}
    h_d(\psi(\mat{A}_r),\psi(\mathat{A}_r)) = \max \left\{\sv \left(\psi(\mat{A}_r),\psi(\mathat{A}_r)\right), \sv \left(\psi(\mathat{A}_r),\psi(\mat{A}_r)\right) \right\}.
\end{equation}
Note that the Hausdorff distance is related to the spectral variation, defined in \eqref{eqn:spvar}, as it measures how close two sets are to each other. We use $h_d$ instead of spectral variation as the spectral variation of two sets depends on the order of the sets in question, while the Hausdorff distance does not.
For both algorithms SVDS-H and RandSVD-H, we computed the Hausdorff distance between the spectra of $\mathat{A}_r$ and $\mat{A}_r$ and $\mathat{A}_r$ and we saw that $h_d \approx  10^{-14}$. 

The final measure of accuracy we will consider for this test problem is the relative error in the Markov parameters. Let $\mat{A}_r,\mat{B}_r,\mat{C}_r$ be the system matrices identified by the full SVD, and let $\mathat{A}_r,\mathat{B}_r,\mathat{C}_r$ be the system matrices identified using another method. The relative error in the Markov parameters is then defined as 
\begin{equation}\label{eqn:markoverr}
    M_k = \frac{\|\mat{C}_r\mat{A}_r^{k}\mat{B}_r - \mathat{C}_r\mathat{A}_r^k \mathat{B}_r\|_2}{\|\mat{C}_r\mat{A}_r^k\mat{B}_r \|_2}, \quad k = 1,\dots,2s-1.
\end{equation}
 We compute this error for the identified matrices using the SVDS-H and RandSVD-H algorithms, and plot the results in Figure~\ref{fig:markoverr_steel1357}. Both algorithms produce comparable error in the Markov parameters, and this error is very low in both cases. These results combined with those from Figure~\ref{fig:evalues_steel1357} show that our algorithms are very accurate compared to the standard algorithms.

\begin{figure}[!ht]
    \centering
    \includegraphics[scale=.4]{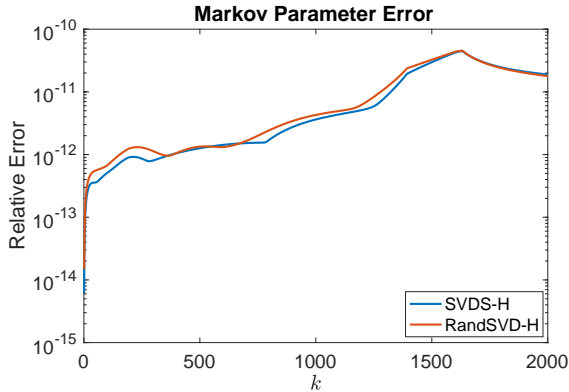}
    \caption{Relative error in recreated Markov parameters compared to a full SVD ERA for both the SVDS-H and RandSVD-H algorithms. We used $s = 1000$ and a target system size of $r = 20$ as inputs.}
    \label{fig:markoverr_steel1357}
\end{figure}

\subsubsection{Timing} We now consider the average run time of the three algorithms on the heat transfer problem. The results, shown in Figure~\ref{fig:time_steel1357}, clearly indicate that using the block Hankel structure as in SVDS-H and RandSVD-H significantly reduces the computational cost, and using a randomized algorithm reduces that cost further, as RandSVD-H is computationally the least expensive of the three. We also see from the plots that the SVDS and RandSVD-H algorithms, which both exploit the block Hankel structure, show approximately linear scaling with $s$ confirming the analysis of the computational costs.

\begin{figure}[!ht]
    \centering
    \includegraphics[scale=.4]{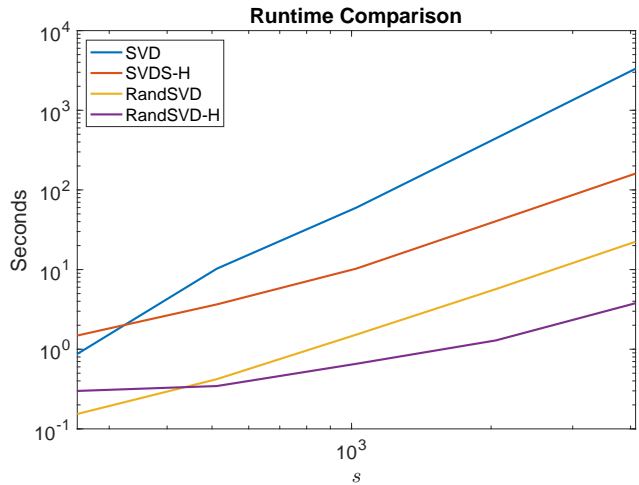}
    \caption{Average run time of the SVD, SVDS, RandSVD, and RandSVD-H algorithms on the heat transfer problem. We averaged the run time over three runs, and a target system size of $r = 20$ as inputs for each algorithm.}
    \label{fig:time_steel1357}
\end{figure}

\subsubsection{Comparison to CUR-ERA} The final test we consider with this application is how CUR-ERA compares to the randomized algorithms we analyzed in the previous experiments. For CUR-ERA, we used the default parameters suggested in the code provided by the authors of \cite{kramer2018system}. The maxvol tolerance is 0.002, and the iterations stop when the relative error of the subsequent iterations is less than $10^{-4}$. For our experiments, we start with the accuracy of the eigenvalues, plotted in Figure~\ref{fig:cur_evals}. 
\begin{figure}[!ht]
    \centering
    \includegraphics[scale=.4]{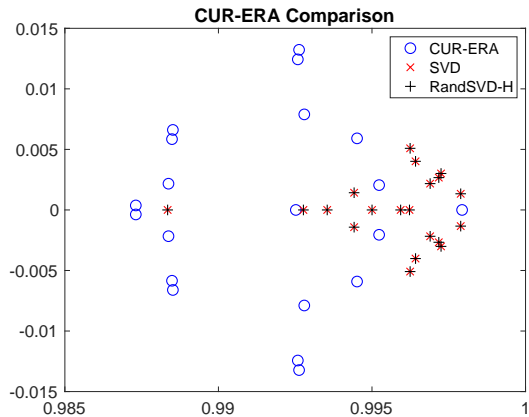}
    \caption{Eigenvalues of the identified $\mat{A}_r$ matrix using CUR-ERA, compared to those identified by SVD and RandSVD-H on the heat transfer problem with $s = 1000$ and a target system size of $r=20$.}
    \label{fig:cur_evals}
\end{figure}
In this figure, we see that the eigenvalues of $\mat{A}$ identified by CUR-ERA are not as accurate as those identified by RandSVD-H. To quantify this, we compute the Hausdorff distance between the spectrum of the $\mat{A}_r$ matrix identified by CUR-ERA and the spectrum of the $\mathat{A}_r$ matrix identified by SVD-ERA when $s = 1000$. We find that this distance is approximately $8.9 \times 10^{-3}$, supporting the claim that CUR-ERA is less accurate than RandSVD-H. At the same time, the CUR-ERA algorithm took $0.83$ seconds compared to $0.66$ seconds for RandSVD-H (averaged over three runs). Finally, we show the relative error in the Markov parameters, computed as defined in \eqref{eqn:markoverr}, for CUR-ERA, RandSVD-H, and SVDS-H. The results are shown in Figure~\ref{fig:cur_merr}, and we see that the error shown here confirms the accuracy comparison, thus far. The accuracy of CUR-ERA  is lower than that obtained from either RandSVD-H or SVDS-H.

\begin{figure}[!ht]
    \centering
    \includegraphics[scale=.35]{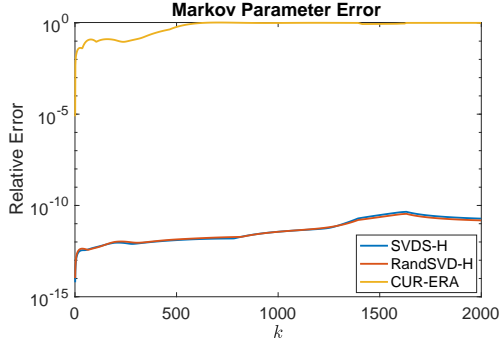}
    \caption{Relative error in recreated Markov parameters compared to a full SVD ERA for RandSVD-H, SVDS-H, and CUR-ERA.}
    \label{fig:cur_merr}
\end{figure}

\subsection{Power system}
For our next application, we consider the system identification of an electric power system model. Identification is an important component of power system modeling as the operating points in a power grid keep changing due to changes in loads, renewable generation profiles, and the network topology, all of which motivate the need for fast algorithms that can quickly update the small-signal model of the grid about these changing operating points from time to time \cite{chakratsg17}. The updated models, in turn, enable operators to make real-time control decisions for ensuring system stability.  For the test problem, we consider a prototype model of reasonably large dimension, namely the IEEE 50-generator model with $145$ buses, $n=155$ state variables, and $m = 50$ input variables, as listed in Table~\ref{tab:datasets}. Details about the physical meaning of the states and model parameters of this test system can be found in \cite{pst}. The original IEEE 50-generator model is nonlinear. The model is, therefore, first linearized about a chosen power flow solution. The resulting LTI model is excited by $50$ impulse functions injected through the {excitation input} terminals of the generators. All $\ell = 155$ states are measured, and are considered as outputs. The output matrix for this example is, therefore, the identity matrix.

An important feature of this problem is that it has a significantly larger number of inputs and outputs with a relatively small system size $n=155$. Since the computational cost of RandSVD-H scales linearly with the product of the number of inputs and outputs, this suggests that RandTERA may have be computationally advantageous if the number of inputs and outputs can be reduced. Therefore, our experiments will compare RandTERA to RandSVD-H both in terms of accuracy and computational costs. For this problem, we again perform system identification as well as model reduction, taking the target rank to be $r = 75$.

\subsubsection{Accuracy} First, we will examine the accuracy of both RandTERA and RandSVD-H on the power system dataset. For all the experiments for testing accuracy, we will use $s=500$, and target rank $r=75$. For the randomized algorithms, we use an oversampling parameter $\rho = 20$, and $q = 1$ subspace iterations. Note that the full Hankel matrix for the accuracy experiments is of size $77500 \times 25000$.

\begin{figure}[!ht]
    \centering
    \begin{subfigure}[b]{.45\textwidth}
    \centering
    \includegraphics[width=\textwidth]{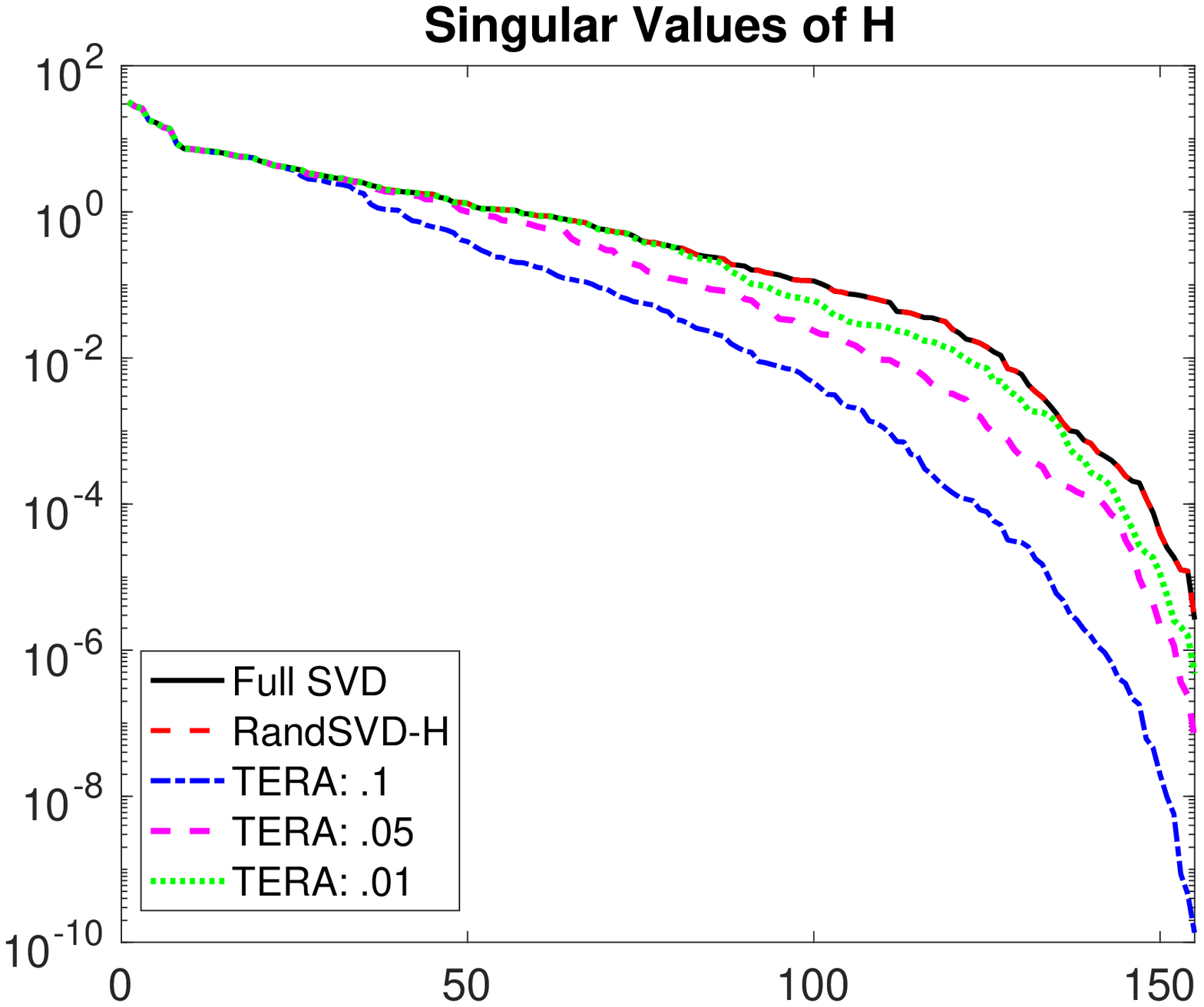}
    \caption{}
    \end{subfigure}
    \begin{subfigure}[b]{.45\textwidth}
    \centering 
    \includegraphics[width=\textwidth]{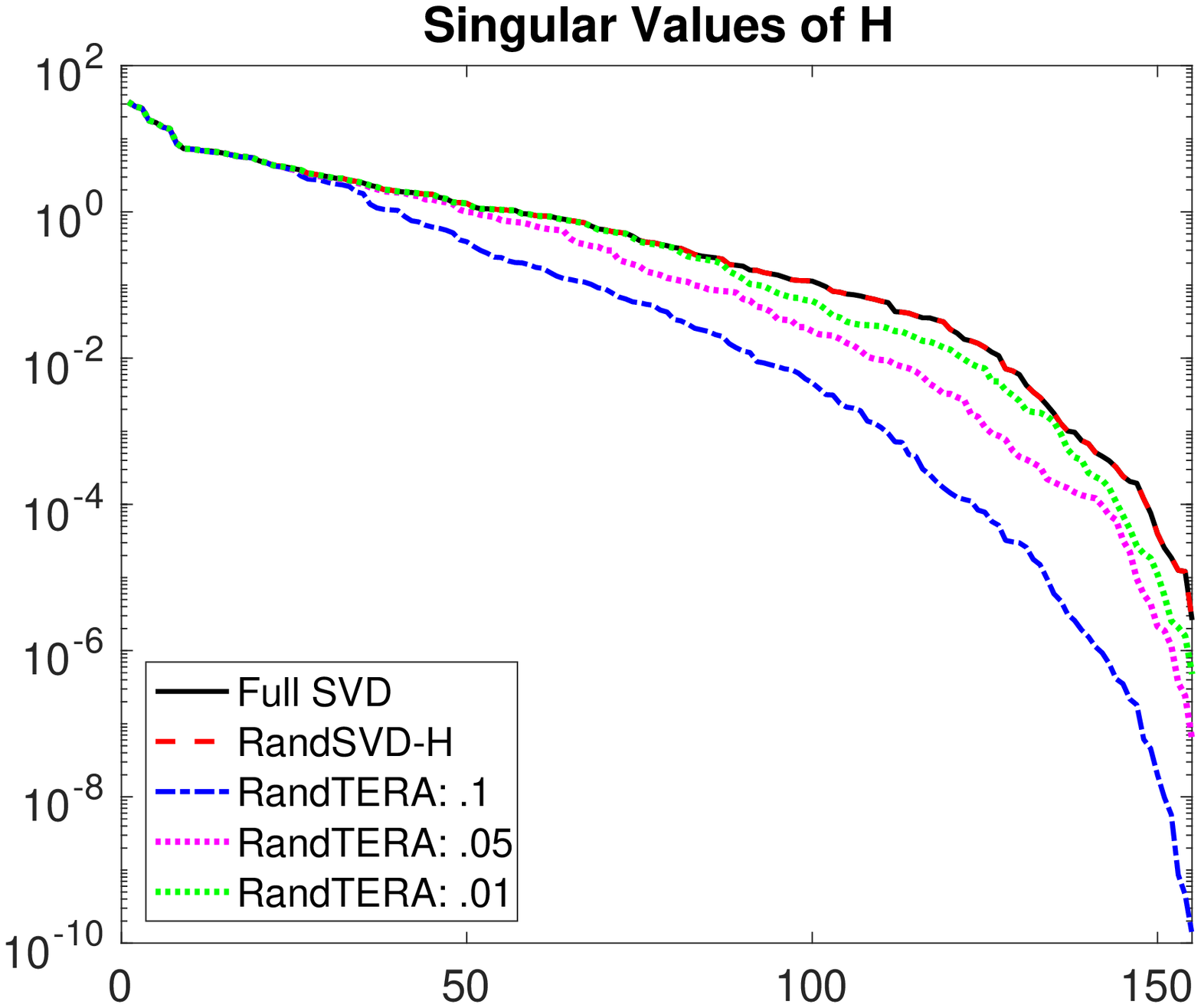}
    \caption{}
    \end{subfigure}
    \caption{Singular values of block Hankel matrix $\matc{H}_s$ as computed by a full SVD and RandSVD-H versus those computed by (a) TERA with varying $\epsilon$ tolerances and (b) RandTERA with varying $\epsilon$ tolerances.}
    \label{fig:svals155_eps}
\end{figure}
\paragraph{Choice of $\ell'$ and $m'$} In our first experiment,  we will investigate the choice of the reduced number of inputs and outputs $\ell'$ and $m'$ for the tangential algorithms (TERA and RandTERA). To choose these parameters, we compute the singular values of $\matc{H}_w$ and $\matc{H}_e$ defined in \eqref{eqn:Hw},\eqref{eqn:He}. The indices $\ell'$ and $m'$ are chosen to be the largest integers such that

\[ \begin{aligned} \ell' = & \arg\min_k \left\{1 \leq  k \leq \ell \left|\sigma_k(\matc{H}_w) \geq \epsilon \sigma_1(\matc{H}_w) \right. \right\}  \\ m' = & \arg\min_k \left\{1 \leq k \leq m \left|{\sigma_k(\matc{H}_e)} \leq \epsilon \sigma_1(\matc{H}_e)\right. \right\}.\end{aligned} \]
That is, they are chosen to be the smallest indices for which the singular values are within a threshold $\epsilon$ of the largest singular value. 
The values of $\ell'$ and $m'$ for different $\epsilon$ thresholds are shown in Table~\ref{tab:randtera_eps}.

\begin{table}[!ht]
    \centering
    \begin{tabular}{c|c|c}
         $\epsilon$ & $m'$ & $\ell'$ \\
         \hline
         $0.1$ & 14 & 18 \\
         $0.05$ & 21 & 33 \\
         $0.01$ & 30 & 59
    \end{tabular}
    \caption{Reduced number of inputs $m'$ and outputs $\ell'$ based on the singular value threshold $\epsilon$ for use in TERA and RandTERA.}
    \label{tab:randtera_eps}
\end{table}

\paragraph{Accuracy of eigenvalues}We first compare the accuracy of the singular values of the computed block Hankel matrix $\matc{H}_s$. First, we show in Figure~\ref{fig:svals155_eps} the effect of different $\epsilon$ values on the accuracy of singular values computed by TERA and RandTERA. We can see that the smaller the tolerance $\epsilon$, the more accurate the singular values. 
Also, the singular values computed using TERA and RandTERA are  less accurate compared to RandSVD-H, but the singular values computed by RandTERA are comparable with those of TERA. This shows that the randomized algorithms are accurate compared to their deterministic counterparts.

\begin{figure}[!ht]
    \centering
    \includegraphics[width=0.8\textwidth]{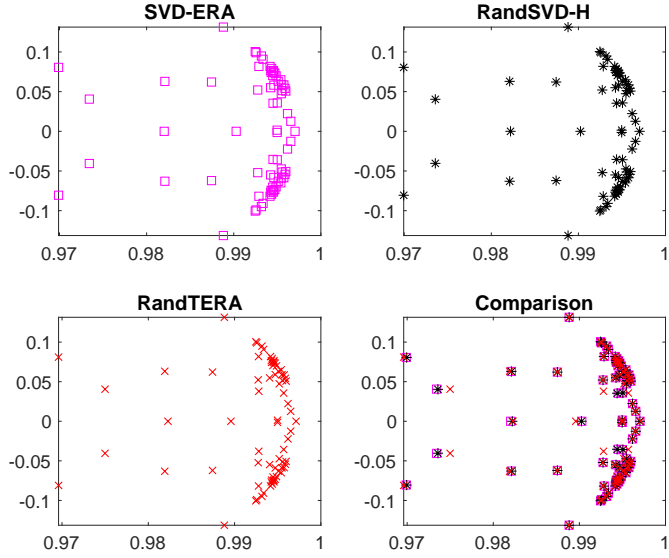}
    \caption{Eigenvalues of the identified  matrix $\mathat{A}_r$ from SVD-ERA (top left panel) compared to the identified  matrix $\mathat{A}_r$ from RandSVD-H (top right) and RandTERA (bottom left panel) on the power system test problem.}
    \label{fig:evalues_155}
\end{figure}

Next, we consider the eigenvalues of the identified $\mathat{A}$ matrix using the RandSVD-H and RandTERA algorithms compared with the eigenvalues of the original $\mat{A}$ matrix. We use the parameters listed above, but we fix $\epsilon = 10^{-2}$. These results are shown in Figure~\ref{fig:evalues_155}. We can see that our algorithms identify the state matrix with eigenvalues very close to the true eigenvalues. To quantify how close they are, we consider the Hausdorff distance again as defined in \eqref{eqn:hausdist}. We computed the Hausdorff distance between the spectrum of the $\mat{A}$ matrix identified by an SVD-ERA and the spectra of the $\mat{A}$ matrices identified by RandSVD-H and RandTERA. We used $s=500$ for all algorithms, and saw that for RandSVD-H, $h_d \approx 2.7 \times 10^{-4}$. For RandTERA, $h_d \approx 3.0 \times 10^{-3}$. 

\begin{figure}[!ht]
    \centering
    \includegraphics[scale=.45]{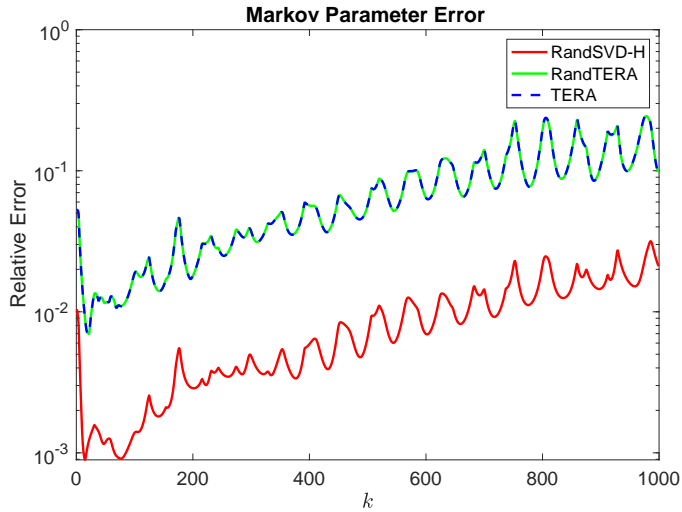}
    \caption{Relative error in the reconstructed Markov parameters using RandTERA, TERA, and RandSVD-H algorithms. The relative error is computed using the system identified using the SVD-ERA.}
    \label{fig:markoverr_155}
\end{figure}
Our final measure of accuracy is the relative error in the Markov parameters as defined in~\eqref{eqn:markoverr}. We plot this relative error  in Figure~\ref{fig:markoverr_155}, and make two observations. The first is that RandSVD-H has a much lower relative error than either TERA or RandTERA. The second is that the error from RandTERA matches the error from TERA very accurately, so the randomization is not causing a significant loss in the accuracy. We can see in Figure \ref{fig:markoverr_155}, that the relative error increases with increasing $k$. This is because the identification techniques using impulse response fail to identify the well damped poles (i.e., the poles whose real part is significantly lesser than 1) accurately, since information corresponding to these modes have to be obtained from the first few Markov parameters only.  In Figure~\ref{fig:timeresp_imp155}, we compare the impulse response of the original full, discrete-time model with that of the models identified using RandSVD-H and RandTERA (an impulse input was fed into Generator 1). The top plot shows the speed response of Generator 1 while the bottom plot shows that for Generator 2. We see that both the identified models have similar performance for Generator 1 and match the original trajectory well, whereas in Generator 2, the trajectories from all  algorithms match each other closely, with only minor deviations from the ``true'' trajectory in the the first $2$ seconds.
 The results for the other generators were observed to follow similar matching trends.

\begin{figure}[!ht]
    \centering
    \includegraphics[width=0.75\textwidth]{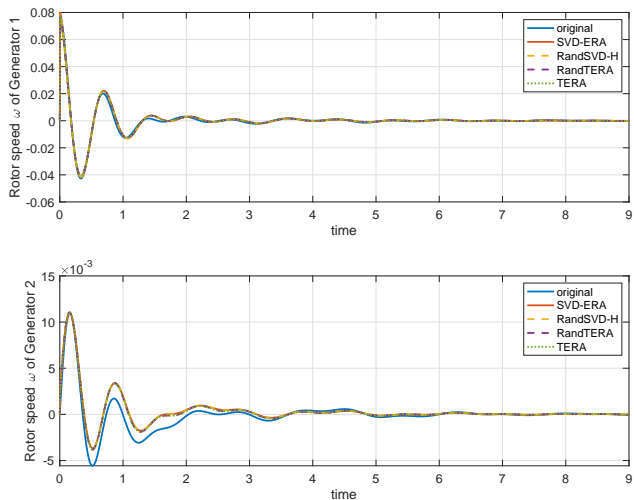}
    \caption{Impulse response of the original discrete-time system compared against that for the systems identified using SVD-ERA, RandSVD-H, TERA, and RandTERA and excited with impulse input at Generator 1, with responses from Generator 1 (top) and Generator 2 (bottom).}
    \label{fig:timeresp_imp155}
\end{figure}

\subsubsection{Timing} 
We next examine the computational cost of RandSVD-H and RandTERA on the power system test problem. We show in Table~\ref{tab:imp155_time} the average CPU time of both of these algorithms compared to a full SVD as $s$ increases. As before, we averaged over three different runs, and used the same parameters for the randomized algorithms, $r=75$, $\rho = 20$ and $q=1$. We see that RandTERA is much cheaper compared to RandSVD-H, and both are much cheaper than a full SVD as expected. Note that for $s = 1000$, the matrix size was $155,000 \times 50,000$ and since this matrix was too large to store explicitly, we do not report the computational cost using the SVD.

\begin{table}[!ht]
    \centering
    \begin{tabular}{c|c|c|c|c|c}
         $s$ & size of $\matc{H}_s$ & SVD & RandSVD-H & TERA & RandTERA  \\
         \hline
         500  & $77,500 \times 25,000$ & 1634.4 & 20.3 & 339.7 & 6.6 \\
         700 & $108,500 \times 35,000$ & 4292.0 & 28.4 & 795.2 & 11.7 \\
         1000 & $155,000 \times 50,000$ & N/A & 41.9 & 1983.9 & 18.0
    \end{tabular}
    \caption{Computational time in seconds of the SVD step in the identification algorithms SVD-ERA, RandSVD-H, TERA, and RandTERA.  Note for $s = 1000$, the matrix was too large to store explicitly, so we only show the run times for the other algorithms.}
    \label{tab:imp155_time}
\end{table}

\section{Conclusions and Future work}
We have presented two different randomized algorithms for accelerating the computational cost of system identification using ERA. The first algorithm accelerates the randomized subspace iteration by efficiently computing matrix vector products using the block Hankel structure. The second algorithm uses the previous algorithm, but on a matrix with a reduced number of inputs and outputs using tangential interpolation. The algorithms no longer have the cubic scaling with $s$, which controls the number of time steps. The error analysis relates the accuracy of the eigenvalues of the  identified system matrix to the accuracy of the singular vectors of the block Hankel matrix $\matc{H}_s$. Numerical experiments from different applications show that our algorithms are both accurate and computationally efficient.

There are several avenues for future work. First, the analysis in Section~\ref{sec:err} is general and does not make assumptions on the specific algorithm used to compute the low-rank approximation to $\matc{H}_s$. The analysis can be specialized by using specific results for the accuracy of the singular vectors, for example, following the approaches in~\cite{nakatsukasa2017accuracy,saibaba2019randomized,simon2000low}. Second, the paper assumes that we have access to the Markov parameters. This is not the case when it is possible to excite the system using general inputs. One possibility, as mentioned earlier, is to use the approach in~\cite{juang1993identification} which may not be feasible when a large number of Markov parameters need to be estimated. Algorithms such as MOESP and N4SID (see~\cite[Chapter 9]{verhaegen2007filtering}) can be used for general inputs but also suffer from high computational costs. However, we can combine the block Hankel structure with randomized algorithms for the general input case as well. Such techniques which involve identification using general input can be extremely accurate in identifying the well-damped poles as they constant excite the system with a persistently exciting signal. This is an ongoing research direction. 

\section{Acknowledgements}
We would also like to thank Eric Hallman for useful comments and for generously sharing his code with us.
\bibliographystyle{siamplain}
\bibliography{ref}

\end{document}